\documentclass[]{amsart}

\usepackage[pdftex]{graphicx}
\usepackage{float}
\usepackage{amsmath, amsthm, amssymb}
\usepackage{mathrsfs}

\theoremstyle{plain}
\newtheorem{theorem}{Theorem}[section]
\newtheorem{proposition}[theorem]{Proposition}
\newtheorem{lemma}[theorem]{Lemma}
\newtheorem{corollary}[theorem]{Corollary}

\theoremstyle{definition}
\newtheorem*{definition}{Definition}
\newtheorem{remark}[theorem]{Remark}

\theoremstyle{remark}
\newtheorem*{acknowledgement}{Acknowledgements}

\title[The colored HOMFLY-PT polynomials of twist knots]
{The colored HOMFLY-PT polynomials of the trefoil knot, the figure-eight knot and twist knots}
\author{Kenichi Kawagoe}
\address{Institute of Liberal Arts and Science \newline
\indent Kanazawa University, Kakuma, Kanazawa, 920-1192, Japan}
\email{kawagoe@se.kanazawa-u.ac.jp}
\date{\today}

\begin{document}

\maketitle

\begin{abstract}
We give a rigorous proof of the colored HOMFLY-PT polynomials 
of the trefoil knot, the figure-eight knot and twist knots.
For the trefoil knot and the figure-eight knot,
it is expressed by a single sum,
and for a twist knot, it is  expressed by a double sum.
\end{abstract}

\section{Introduction}
After the articles \cite{Kawagoe3,Kawagoe4}, 
the author feels that an explicit formula of the colored HOMFLY-PT polynomial 
of the figure-eight knot is desirable \cite{AV}.
Actually, in \cite{Kawagoe3}, we devote to calculate the colored 
HOMFLY-PT polynomial of the figure-eight knot at a root of unity,
which is expressed by a single sum. 
In \cite{Kawagoe4}, we treat  twist knots,
which is expressed by more than double sum depending on the number of full twists.
Other expressions 
are also conjectured for 
the figure-eight knot \cite{IMMM} and the twist knot \cite{NRZS},
which is expressed by as a single sum and a double sum, respectively.
Therefore, in this article, we give a rigorous proof of the colored 
HOMFLY-PT polynomial of twist knots.
This article is a generalization
of the colored Jones polynomials \cite{H,Le,Masbaum}. 

The proof is similar as \cite{Masbaum}.
That is, for a positive integer $n$, we construct an element $\omega=\omega_{n}^{+}$ 
of the HOMFLY-PT skein module
satisfying Figure \ref{fig:full twist},
where the horizontal arc consists of the two $i$-th $q$-symmetrizers
with opposite orientations  ($i=0,\ldots,n$):
\begin{figure}[H]
\centering
\includegraphics[width=50mm,pagebox=cropbox,clip]{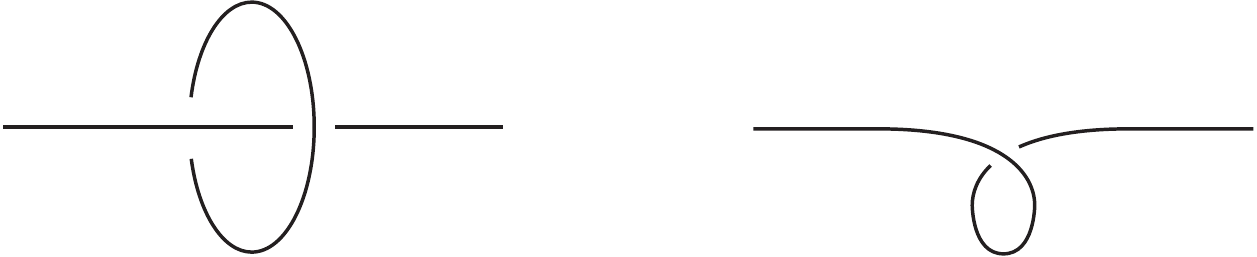}
\begin{picture}(0,0)(0,0)
\put(-78,12){$=$}
\put(-112,30){$\omega$}
\end{picture}
\caption{$\omega$ and a positive full twist}
\label{fig:full twist}
\end{figure}

This article is organized as follows.
In Section $2$,  we recall some notations and notions concerning the skein module
derived from the HOMFLY-PT skein relations.
In Section $3$,  we develop some eigenvalues of an element of 
the skein module. 
Using the eigenvalues, we define the element $\omega$.
In Section $4$, for twist knots, we discuss a formula of  $p$ full twists.
In Section $5$, we calculate the colored HOMFLY-PT polynomial of twist knots.
As a corollary, for the trefoil and the figure-eight knot,
we can express it  as a single sum. 

Although we only treat twist knots in this article, we hope that
the twisting formula of Section $4$ leads an application to
the colored HOMFLY-PT polynomial of knots which contain twisting arcs. 

\section{Preliminaries}
In this section, 
we review  some notation
including $q$-integers and $q$-symmetrizers.
Let $a$ and $q$ be non-zero variables in $\mathbb{C}$.
For an integer $n$, we define the symbols by
\begin{align*}
[n] &= \frac{q^n-q^{-n}}{q-q^{-1}}, \qquad 
\{ n \} =  q^{n} - q^{-n}, \qquad
\{ n ; a \} =  aq^{n}-a^{-1}q^{-n}.   
\end{align*}
For integers $n >0 , i \geq 0$, we define the products of  $i$ terms of these symbols by
\begin{align*}
[n]_{i} &= [n] [n-1] \cdots [n-i+1], \\
\{ n \}_{i} &=  \{ n \} \{ n-1 \} \cdots \{ n-i+1 \},  \\
\{ n ; a \}_{i} &=  \{ n; a \} \{ n-1; a \} \cdots \{ n-i+1; a \},  \\ 
\{ -n ; a \}_{i} &=  \{ -n; a \} \{ -n+1; a \} \cdots \{ -n+i-1; a \}, 
\end{align*} 
where they are defined to be $1$ if $i = 0$.
Furthermore, we define
\begin{align*}
[n]! =[n]_{n}, \quad \{ n \} ! =  \{ n \}_{n}, \quad 
\left[ \begin{array}{@{\,}c@{\,}}n \\ i \end{array} \right]
= \frac{[n]!}{[i]! [n-i]!}.
\end{align*}


We recall the HOMFLY-PT skein module, $\mathcal{S}(M)$, 
of an oriented $3$-manifold $M$ \cite{HOMFLY, Li}.
$\mathcal{S}(M)$ is 
the free $\mathbb{C}$-module generated by  isotopy classes of framed links in $M$
modulo the  submodule generated by the HOMFLY-PT skein relations:
\begin{enumerate}
\setlength{\itemsep}{3mm}
\item  $L\cup (\textrm{a trivial knot with 0 framing}) = \displaystyle \frac{ \{n ; a \} }{ \{1\} } L$, 
 and $\varnothing = 1$, 
\item \begin{minipage}{15pt}
        \begin{picture}(15,15) 
            \qbezier(0,15)(0,15)(15,0)
            \qbezier(0,0)(0,0)(6,6)
            \qbezier(9,9)(9,9)(15,15)
            \put(15,0){\vector(1,-1){0}}
            \put(15,15){\vector(1,1){0}}
        \end{picture}
\end{minipage}
$\, - \,$
\begin{minipage}{15pt}
        \begin{picture}(15,15) 
            \qbezier(0,0)(0,0)(15,15)
            \qbezier(0,15)(0,15)(6,9)
            \qbezier(9,6)(9,6)(15,0)
            \put(15,0){\vector(1,-1){0}}
            \put(15,15){\vector(1,1){0}}
        \end{picture}
\end{minipage}
$\, = (q-q^{-1}) \,$
\begin{minipage}{30pt}
        \begin{picture}(15,15) 
            \qbezier(0,15)(7.5,7.5)(15,15)
            \qbezier(0,0)(7.5,7.5)(15,0)
            \put(15,0){\vector(3,-2){0}}
            \put(15,15){\vector(3,2){0}}
        \end{picture}
\end{minipage}\hspace*{-3mm}, 
\item \begin{minipage}{18pt}
        \begin{picture}(18,15) 
            \qbezier(0,11.2)(10.5,9)(10.5,4.5)
            \qbezier(10.5,4.5)(10.5,0)(7.5,0)
            \qbezier(7.5,0)(4.5,0)(4.5,4.5)
            \qbezier(4.5,4.5)(4.5,5.2)(6,7.5)
            \qbezier(9,9.7)(10.5,11.2)(15,11.2)
            \put(18,12.2){\vector(4,1){0}}
        \end{picture}
\end{minipage}
$\, = a \,$
\begin{minipage}{18pt}
        \begin{picture}(15,15) 
            \qbezier(0,7.5)(0,7.5)(15,7.5)
            \put(16.5,7.5){\vector(1,0){0}}
        \end{picture}
\end{minipage}
$\, , \quad$
\begin{minipage}{18pt}
        \begin{picture}(18,15) 
            \qbezier(0,11.2)(4.5,11.2)(6,9.7)
            \qbezier(9,7.5)(10.5,5.2)(10.5,4.5)
            \qbezier(10.5,4.5)(10.5,0)(7.5,0)
            \qbezier(4.5,4.5)(4.5,0)(7.5,0)
            \qbezier(4.5,4.5)(4.5,9)(15,11.2)
            \put(18,12.2){\vector(4,1){0}}
        \end{picture}
\end{minipage}
$\; = a^{-1} \,$
\begin{minipage}{18pt}
        \begin{picture}(15,15) 
            \qbezier(0,7.5)(0,7.5)(15,7.5)
            \put(16.5,7.5){\vector(1,0){0}}
        \end{picture}
\end{minipage}$\, .$
\end{enumerate}
We call a crossing \textit{positive} or \textit{negative} 
if it is the same crossing as the  first or second terms in (ii), respectively.

Following \cite{Kawagoe1,Kawagoe2,Kawagoe4}, we review the definitions of the $q$-symmetrizer, 
the $q$-antisymmetrizer, and their properties.
For an integer $n \geq 1$, we recursively define the $n$-th $q$-symmetrizer and
the $n$-th $q$-antisymmetrizer by Figure \ref{fig:qsym},
where the $q$-symmetrizer  is denoted by a white rectangle
and the $q$-antisymmetrizer is denoted by a  black rectangle.
\begin{figure}[H]
\begin{picture}(0,0)(0,0)
\put(23,99){\scriptsize$1$}  \put(40,93){$=$}
\put(23,71){\scriptsize$n$}  \put(40,66){$\displaystyle = \; \frac{ q^{-n+1}  }{ [n]  }$}  
\put(105,75){\scriptsize$n-1$}  
\put(125,66){$\displaystyle+ \frac{ [n-1]  }{ [n]}$}
\put(184,80){\scriptsize$n-2$}  \put(210,75){\scriptsize$n-1$}  
\put(23,43){\scriptsize$1$} \put(40,39){$=$}
\put(23,15){\scriptsize$n$}  \put(40,10){$\displaystyle = \; \frac{ q^{n-1}  }{ [n]  }$}  
\put(105,18){\scriptsize$n-1$} 
\put(125,10){$\displaystyle- \frac{ [n-1]  }{ [n]}$}
\put(184,24){\scriptsize$n-2$}  \put(210,19){\scriptsize$n-1$}
\end{picture}
\centering
\includegraphics[width=75mm,pagebox=cropbox,clip]{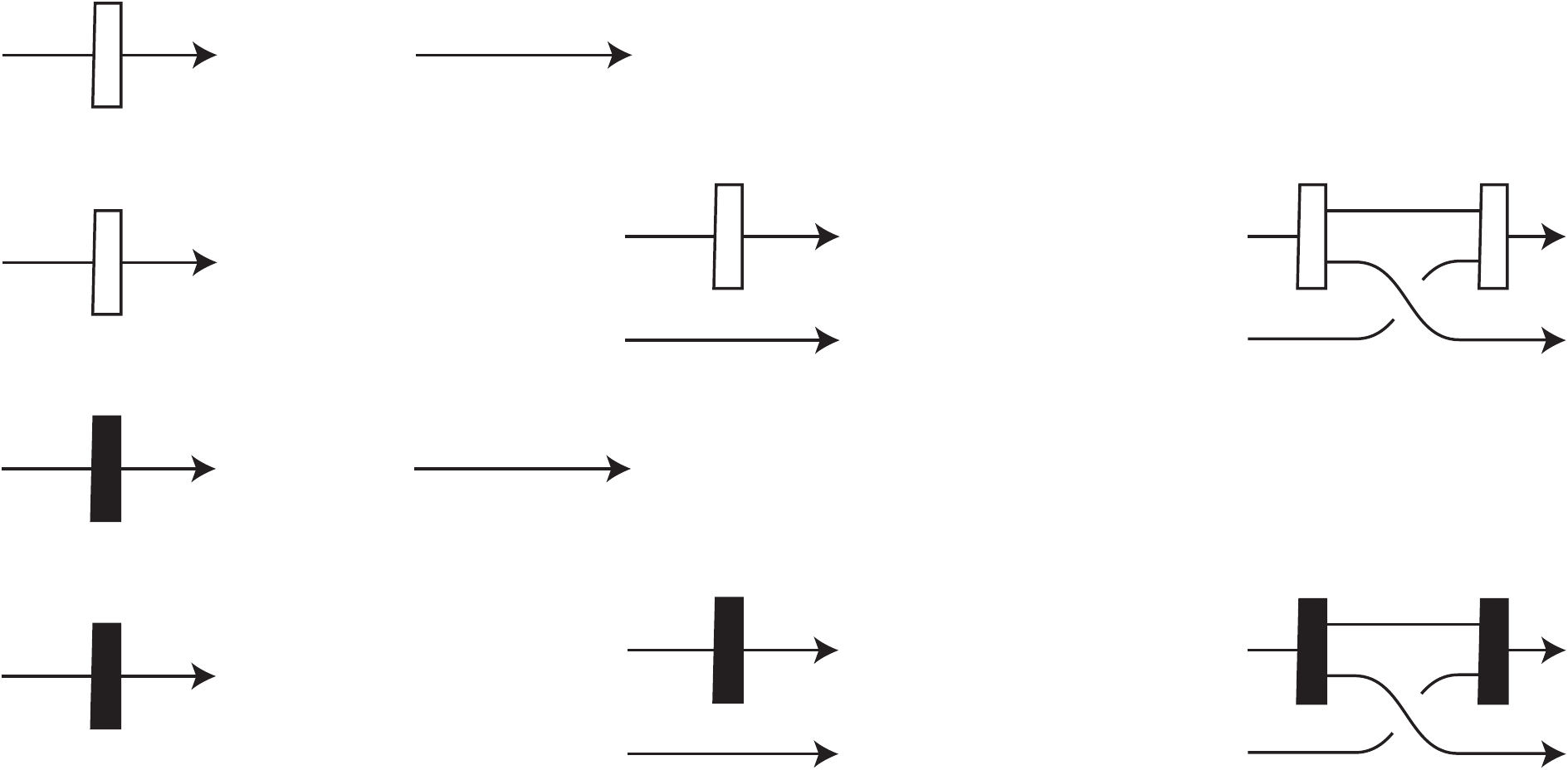}
\caption{The definitions of symmetrizers}
\label{fig:qsym}
\end{figure}
\noindent
The integer $n$ beside an arc means $n$ copies of the arc.
When we can easily find the number of copies, we omit it.

The $q$-symmetrizer and the $q$-antisymmetrizer  have the following properties,
where the positive crossing in Figure \ref{fig:qsymp} means only one crossing:
\begin{figure}[H]
\begin{picture}(0,0)(0,0)
\put(8,62){\scriptsize$n$} 
\put(38,57){$=$}  
\put(72,62){\scriptsize$n$}  
\put(87,57){$= q$}
\put(126,62){\scriptsize$n$} 
%
\put(196,62){\scriptsize$n-1$}  
\put(215,57){$\displaystyle = \; \frac{ \{ n-1 ;a \}  }{ \{ n \}  }$}  
\put(295,62){\scriptsize$n-1$} 
\put(8,19){\scriptsize$n$}
\put(38,14){$=$} 
\put(72,19){\scriptsize$n$}
\put(87,14){$=-q^{-1}$}
\put(140,19){\scriptsize$n$} 
%
\put(196,19){\scriptsize$n-1$} 
\put(215,14){$\displaystyle = \; \frac{ \{ -n+1 ;a \}  }{ \{ n \}  }$}  
\put(304,19){\scriptsize$n-1$} 
\end{picture}
\centering
\includegraphics[width=110mm,pagebox=cropbox,clip]{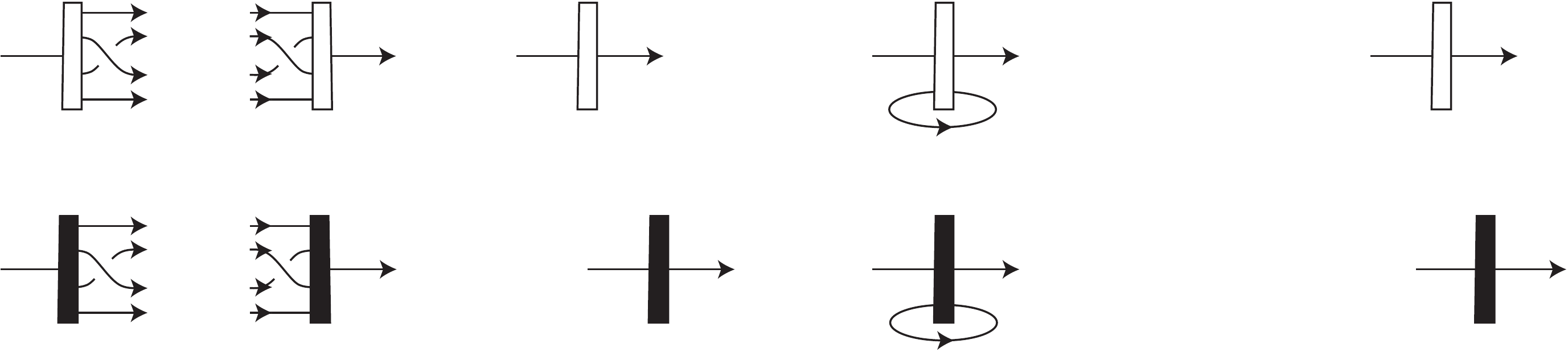}
\caption{Properties of symmetrizers}
\label{fig:qsymp}
\end{figure}

Using the $m$-th and $n$-th $q$-symmetrizers, 
we define the $(m,n)$-th  $q$-symmetrizer by
\begin{figure}[H]
\centering
\includegraphics[width=55mm,pagebox=cropbox,clip]{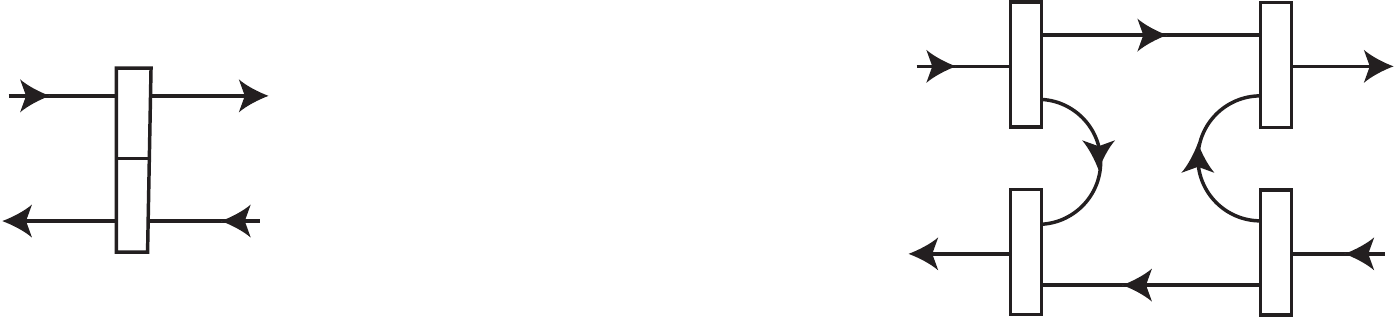}
\begin{picture}(0,0)(0,0)
\put(-125,15){$\displaystyle =  \sum_{i=0}^{\min\{m,n\} }  x_{m,n}^i$}
\put(-158,28){\text{\scriptsize $m$}} \put(-140,28){\scriptsize $m$}
\put(-158,4){\scriptsize$n$}  \put(-140,4){\scriptsize$n$}
\put(-55,31){\scriptsize$m$} \put(-12,31){\scriptsize$m$}
\put(-55,2){\scriptsize$n$}  \put(-12,2){\scriptsize$n$}
\put(-40,35){\scriptsize$m-i$} \put(-40,-2){\scriptsize$n-i$}
\put(-36,23){\scriptsize$i$}  \put(-26,8){\scriptsize$i$}
\end{picture}
\caption{The definition of the $(m,n)$-th  $q$-symmetrizer}
\label{fig:Dmn0}
\end{figure}
\noindent
where $x_{m,n}^i$ is given by
\begin{align*}
x_{m,n}^i = (-1)^i \left[ \begin{array}{@{\,}c@{\,}}m \\ i \end{array} \right]
\left[ \begin{array}{@{\,}c@{\,}}n \\ i \end{array} \right]
\frac{ \{ i \}!  }{  \{ m+n -2;a \}_i  }.
\end{align*}

These three symmetrizers 
have idempotent properties and vanishing properties as in Figure \ref{fig:vanishing property}.
\begin{figure}[H]
\begin{picture}(0,0)(0,0)
\put(38,97){$=$}  \put(89,97){$=$}  
%
\put(194,97){$=$}  \put(245,97){$=$}  
\put(141,54){$=$} 
\put(85,11){$=$}  \put(129,11){$= 0$}
\put(199,11){$=$} \put(241,11){$= 0$}
\end{picture}
\centering
\includegraphics[width=100mm,pagebox=cropbox,clip]{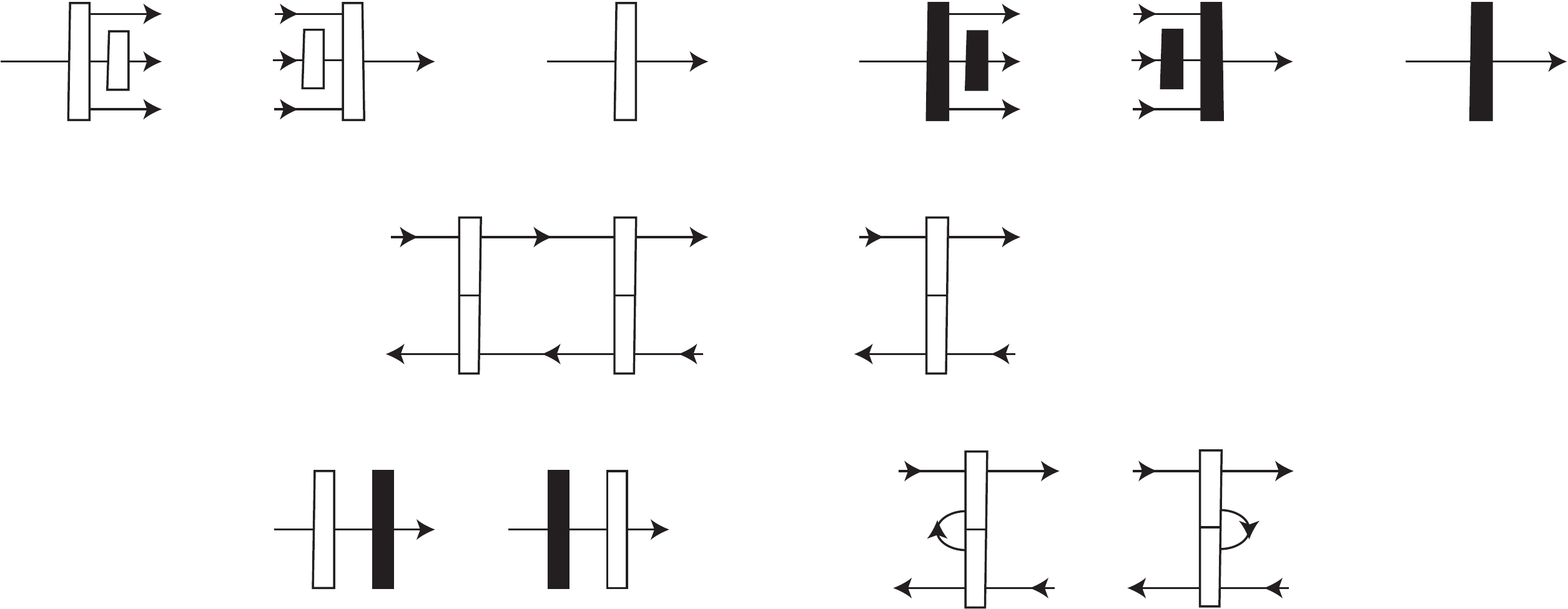}
\caption{Idempotent and vanishing properties}
\label{fig:vanishing property}
\end{figure}

The next three lemmas are useful for calculations 
of the HOMFLY-PT skein module in Section \ref{Eigenvalues of $D_{n,n}$}.
\begin{lemma}[\cite{Kawagoe4}]\label{lemma:alpha}
The following hold.
\begin{figure}[H]
\centering
\includegraphics[width=75mm,pagebox=cropbox,clip]{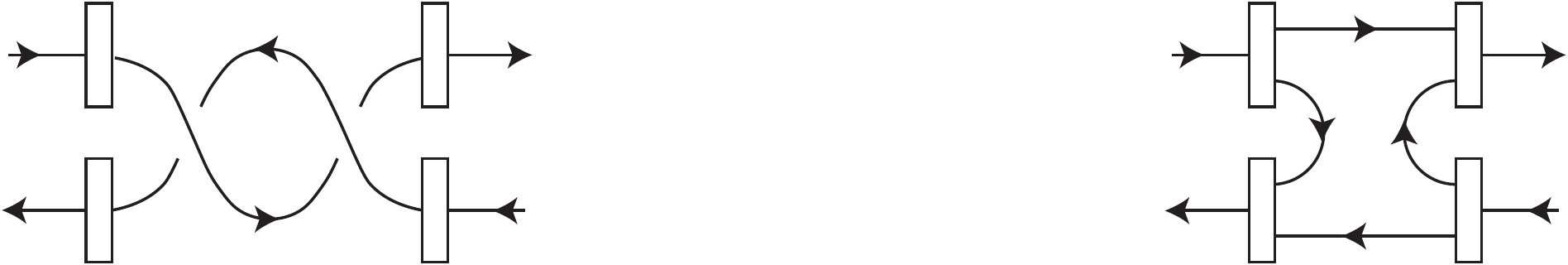}
\begin{picture}(0,0)(0,0)
\put(-140,14){$\displaystyle =  \sum_{i=0}^{\min\{m,n\} }  \alpha_{m,n}^i$}
\put(-215,31){\scriptsize $m$} \put(-153,31){\scriptsize $m$}
\put(-215,1){\scriptsize $n$}  \put(-153,1){\scriptsize $n$}
\put(-55,32){\scriptsize $m$} \put(-13,32){\scriptsize $m$}
\put(-55,2){\scriptsize $n$}  \put(-13,2){\scriptsize  $n$}
\put(-39,35){\scriptsize $m-i$} \put(-39,-3){\scriptsize $n-i$}
\put(-35,21){\scriptsize $i$}  \put(-27,10){\scriptsize $i$}
\end{picture}
\caption{}
\label{fig:alpha}
\end{figure}
\noindent
where $ \alpha_{m,n}^i =  \alpha_{m,n}^i(a,q)$ is given by 
\begin{align*}
 \alpha_{m,n}^i = 
(-a)^{-i} q^{-i (m+n) +\frac{i(i+3)}{2}}
\frac{  \{ m \}_i \{ n \}_i    }{   \{ i \} !        }.
\end{align*}
\end{lemma}

In the same way as  Lemma \ref{lemma:alpha}, 
by applying the vanishing property of $H_m$ and $E_n$, we have the next lemma.
\begin{lemma}\label{lemma:antialpha}
The following hold.
\begin{figure}[H]
\centering
\includegraphics[width=100mm,pagebox=cropbox,clip]{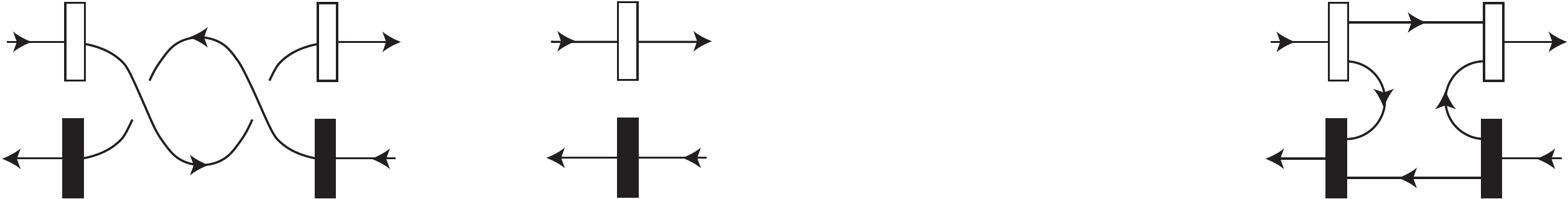}
\begin{picture}(0,0)(0,0)
\put(-286,30){\scriptsize $m$} \put(-225,30){\scriptsize $m$}
\put(-286,0){\scriptsize $n$}  \put(-225,0){\scriptsize  $n$}
\put(-205, 15){$=$}
\put(-185,30){\scriptsize $m$}  \put(-168,0){\scriptsize $n$}
\put(-145,15){$ \displaystyle - a^{-1}q^{m-n}\frac{  \{ m \} \{ n\}   }{  \{1\}  }$}
\put(-55,30){\scriptsize $m$} \put(-13,30){\scriptsize $m$}
\put(-55,1){\scriptsize $n$}  \put(-13,1){\scriptsize $n$}
\put(-40,33){\scriptsize $m-1$} \put(-40,-3){\scriptsize $n-1$}
\end{picture}
\caption{}
\label{fig:antialpha}
\end{figure}
\end{lemma}

\begin{remark}
In Figure \ref{fig:alpha} and \ref{fig:antialpha},
when all the negative crossings change into positive crossings,
the coefficients also change so that $a, q$ are replaced by $a^{-1}, q^{-1}$, respectively.
For our convenience, we set $\bar{\alpha}_{m,n}^i=\alpha_{m,n}^i(a^{-1},q^{-1})$. 
\end{remark}

\begin{lemma}[\cite{Kawagoe4}]\label{lemma:beta}
For integers $ 0 \leq i  \leq  j \leq \min\{m,n \}$, we have
\begin{figure}[H]
\centering
\includegraphics[width=75mm,pagebox=cropbox,clip]{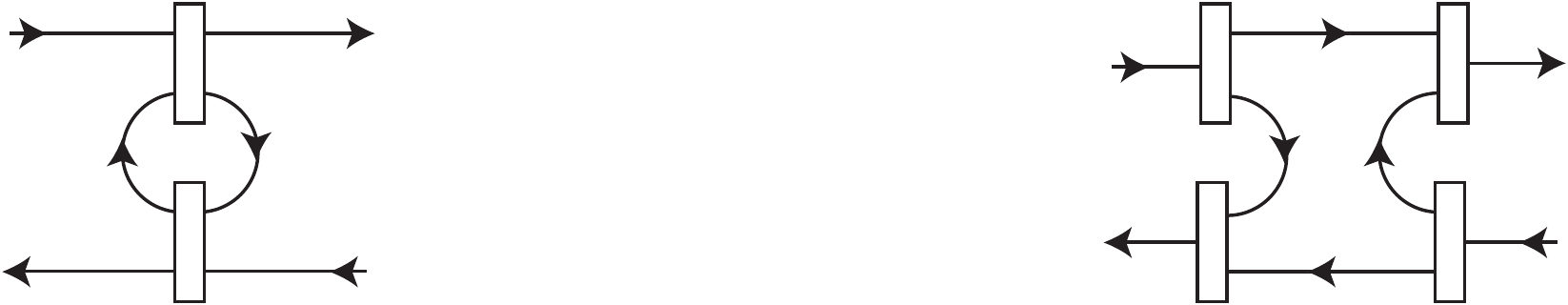}
\begin{picture}(0,0)(0,0)
\put(-214,40){\scriptsize $m-i$} \put(-185,40){\scriptsize $m-j$}
\put(-208,20){\scriptsize $i$}  \put(-178,20){\scriptsize $j$}
\put(-214,-2){\scriptsize $n-i$}  \put(-185,-2){\scriptsize $n-j$}
\put(-155,18){$\displaystyle =   \sum_{k=0}^{i} \; \beta_{i,j:m,n}^k$}
\put(-73,36){\scriptsize $m-i$} \put(-15,36){\scriptsize $m-j$}
\put(-73,1){\scriptsize $n-i$}  \put(-15,1){\scriptsize $n-j$}
\put(-52,42){\scriptsize $m-j-k$} \put(-52,-5){\scriptsize $n-j-k$}
\put(-78,18){\scriptsize $j-i+k$}  \put(-32,12){\scriptsize $k$}
\end{picture}
\caption{}
\label{fig:beta}
\end{figure}
\noindent
where  $\beta_{i,j:m,n}^k =  \beta_{i,j:m,n}^k(a,q)$ is given by
\begin{align*}
 \beta_{i,j:m,n}^k = 
\frac{  \{ m -j \}_k   }{   \{ m \}_i    }   \frac{  \{ n -j \}_k   }{   \{ n \}_i    } \{ j  \}_{i-k}
\left[ \begin{array}{@{\,}c@{\,}}i \\ k \end{array} \right]
 \{ m + n -j -k -1 ; a \}_{i-k}.
\end{align*}
\end{lemma}

\begin{remark}
In Figure \ref{fig:alpha} and Figure \ref{fig:beta}, 
when all the $q$-symmetrizers change into $q$-antisymmetrizers,
the coefficients also change so that $q$ is replaced by $-q^{-1}$.
\end{remark}

The skein module of the solid torus $S^1 \times D^2$ is denoted by $\mathcal{S}$.
Here, we consider  submodules of  $\mathcal{S}$.
For a circle along $S^1$ of the solid torus,
we define $H_n \in \mathcal{S}$ by $n$ copies of the circles 
inserted by the $n$-th $q$-symmetrizer.
Similary, we define $E_n$ by the $n$-th $q$-antisymmerizer.
Each symmetrizer is compatible with 
the orientation of the circle.
We define $H_{n,n}$ by two copies of $H_n$,
one is anticlockwise and the other is clockwise.
For two circles along $S^1$, one is anticlockwise and the other is clockwise,
we define $D_{m,n}$ by $m$ copies of  the anticlockwise circle and
$n$ copies of the clockwise inserted by the $(m,n)$-th $q$-symmetrizer.

Let $\mathcal{H}_{n,n}$ and $\mathcal{D}_{n,n}$
be the submodules spaned by $H_{i,i}$ and ${D}_{i,i}$ for $i=0,\ldots,n$, respectively.
Then,  Lemma \ref{lemma:beta} implies that
$\mathcal{D}_{n,n}$ is spaned by $H_{i,i}$ for $i=0,\ldots,n$, and vice versa.
Hence, we have $\mathcal{D}_{n,n} = \mathcal{H}_{n,n}$.

Let $\langle \;  \; \rangle$ be the linear map on  $\mathcal{S}$ to  $\mathbb{C}$
defined by  evaluating it in $S^3$.
Let $t : \mathcal{S} \to \mathcal{S}$ be the twist map 
induced by one right-handed twist on the solid torus as shown 
in the right-hand side of Figure \ref{fig:full twist},
where the twist is induced at the bottom of the solid torus.
Similarly, let $t^{-1}$ be the twist map induced by one left-hand twist.
For $x \in \mathcal{S}$, let $e_{x} : \mathcal{S} \to \mathcal{S}$ be
the map encircling an element of $\mathcal{S}$ by $x$ as shown 
in the left-hand side of Figure \ref{fig:full twist}, 
where $x$ slides into the bottom of the solid torus and 
encircles the element. 
In Section \ref{Eigenvalues of $D_{n,n}$}, 
we find two elements $\omega_n^+$ and $\omega_n^-$
so that $e_{\omega_n^{\pm}}(D_{n,n}) =t^{\pm 1}(D_{n,n})$.
We give some examples of these maps:
\begin{align*}
\langle H_n \rangle &= \frac{ \{n-1;a\}_n  }{  \{n\}!  }, & 
\langle E_n \rangle &= \frac{ \{-n+1;a\}_n  }{  \{n\}!  }, \\
t(H_n) &= a^{n}q^{n(n-1)} H_n, & t(E_n) &=a^{n}(-q)^{-n(n-1)} E_n, \\
t(D_{m,n}) &= a^{m+ n}q^{m(m-1)+n(n-1)} D_{m,n}. & &
\end{align*}

\begin{lemma}\label{lemma:D_{m,n}}
For positive integers $m \geq n $, we have the following:
\begin{align*}
\langle D_{m,n} \rangle =\frac{ \{ m + n-1; a \} \{ m-2;a \}_{m-1} \{ n-2;a \}_{n-1} \{-1;a\}  }{   \{ m \}! \{ n \}! }.
\end{align*}
\end{lemma}
\begin{proof}

Since $\langle D_{m,n} \rangle$ is described as follows,
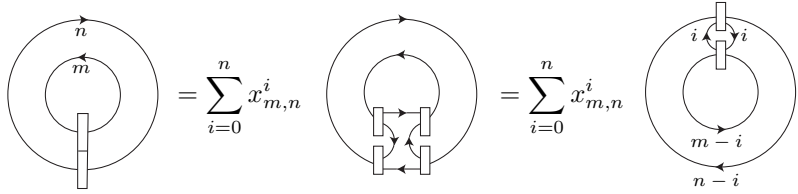
\begin{figure}[H]
\begin{picture}(0,0)(0,0)
\put(28,44){\scriptsize $m$} \put(29,58){\scriptsize $n$}
\put(68,33){$\displaystyle =   \sum_{i=0}^{n} x_{m,n}^i$}
\put(190,33){$\displaystyle =   \sum_{i=0}^{n} x_{m,n}^i$}
\put(262,56){\scriptsize $i$}\put(281,56){\scriptsize $i$}
\put(262,16){\scriptsize $m-i$} \put(263,1){\scriptsize $n-i$}
\end{picture}
\centering
\includegraphics[width=105mm,pagebox=cropbox,clip]{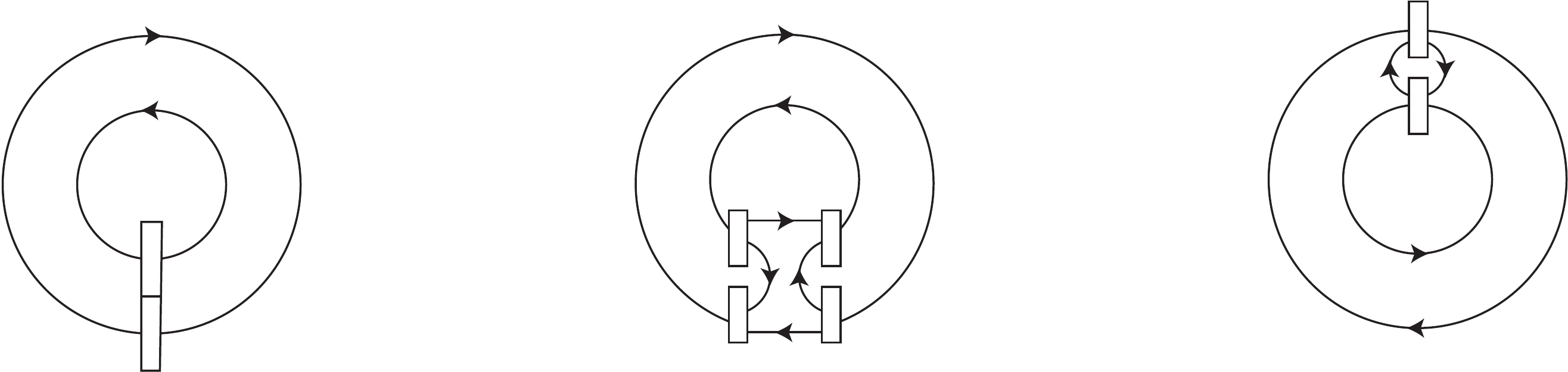}
\caption{The calculation of $\langle D_{m,n} \rangle$}
\label{fig:Dmn1}
\end{figure}
\noindent
$\langle D_{m,n} \rangle$ is given by
\begin{align*}
&\sum_{i=0}^n x_{m,n}^i \frac{\{m-1;a\}_{m-i}}{ \{m\}_{m-i}  } 
\frac{\{n-1;a\}_{n-i}}{ \{n\}_{n-i}   }  \frac{\{i-1;a\}_{i}}{ \{ i \}!   }\\
=&
\sum_{i=0}^n (-1)^i 
\frac{  \{i-1;a \}_{i}  }{  \{m+n-2;a\}_{i}  }
\frac{  \{m-1;a \}_{m-i}  }{  \{m-i\}!  }
\frac{  \{n-1;a \}_{n-i}  }{  \{n-i\}!  } \\
=&
\frac{ \{m-1;a \}_m  }{ \{ m+n-2;a \}_n \{m\}! \{ n \}! } \\
&\quad \times \sum_{i=0}^n (-1)^i  \{m\}_i \{n\}_i \{m+n-i-2;a\}_{n-i} \{n-1;a\}_{n-i}.
\end{align*}
Here, we use the following transformation:
\begin{align*}
\{m-1;a\} \{n-1;a\} =\{m+n-i-1;a \} \{i-1;a\} + \{ m-i \} \{ n-i \}.
\end{align*}
Then, each term in the sum is expressed by 
\begin{align*}
&\{m\}_i \{n\}_i \{m+n-i-2;a\}_{n-i} \{n-1;a\}_{n-i} \\
=& \{ m \}_{i} \{n\}_{i}  \{m+n-i-2;a\}_{n-i-1}  \{ m-1;a\} \{ n-1;a \}  \{n-2;a\}_{n-i-1} \\
=& \{ m \}_{i} \{n\}_{i}  \{m+n-i-1;a\}_{n-i} \{n-2;a\}_{n-i} \\
&\quad + \{ m \}_{i+1} \{ n \}_{i+1}  \{m+n-i-2;a\}_{n-i-1}  \{n-2;a\}_{n-i-1}.
\end{align*}
Hence, we have
\begin{align*}
&\sum_{i=0}^n (-1)^i  \{m\}_i \{n\}_i \{m+n-i-2;a\}_{n-i} \{n-1;a\}_{n-i} \\
= &
\sum_{i=0}^{n-1} (-1)^i \Bigl(
 \{ m \}_{i} \{n\}_{i}  \{m+n-i-1;a\}_{n-i} \{n-2;a\}_{n-i} \\
&\quad + \{ m \}_{i+1} \{ n \}_{i+1}  \{m+n-i-2;a\}_{n-i-1}  \{n-2;a\}_{n-i-1} 
\Bigr) \\
&\quad +  (-1)^n   \{ m \}_{n} \{n\}_{n} \\
=&
 \{m+n-1;a\}_{n} \{n-2;a\}_{n}.
\end{align*}
Therefore, we obtain
\begin{align*}
\langle D_{m,n} \rangle  
&= \frac{ \{m-1;a \}_m  }{ \{ m+n-2;a \}_n \{m\}! \{ n \}! }  \{m+n-1;a\}_{n} \{n-2;a\}_{n} \\
&= \frac{ \{ m + n-1; a \} \{ m-2;a \}_{m-1} \{ n-2;a \}_{n-1} \{-1;a\}  }{   \{ m \}! \{ n \}! }.
\end{align*}
\end{proof}

\section{Eigenvalues of $D_{n,n}$}\label{Eigenvalues of $D_{n,n}$}
In this section, we find several eigenvalues of  $D_{n,n}$ concerning the encircling map $e$
and the twist map $t$.
\subsection{The eigenvalue for the encircling map of $H_i$}
\begin{lemma}\label{lemma:sigma}
For an integer $i \geq 0$,
we have $e_{H_i}(D_{n,n})=\sigma_{n,i}D_{n,n}$, where
\begin{align*}
\sigma_{n,i} &= 
\sum_{\substack{ 0 \leq j+k \leq i \\ j < k }}  
\Bigl(
(-1)^j
 a^{-j+k} q^{\epsilon_{j,k}} 
+(-1)^k a^{-k+j} q^{\epsilon_{k,j}}
\Bigr) \frac{ \{n\}_j \{ n \}_k     }{  \{ i -j - k \}!  } \{ i-1;a \}_{i-j-k}               \\ 
&\qquad + \sum_{j=0}^{\lfloor \frac{i}{2} \rfloor } 
(-1)^j \frac{ \{n\}_j \{ n \}_j     }{  \{ i -2 j  \}!  } \{ i-1;a \}_{i- 2 j}, \\
\epsilon_{j,k} &={(-j+k)(i+n) +\frac{j(j+3)}{2} - \frac{k(k+3)}{2}}.
\end{align*}
\end{lemma}

\begin{proof}
 $e_{H_i}(D_{n,n})$ is described by
\begin{figure}[H]
\begin{picture}(0,0)(0,0)
\put(38,213){\scriptsize $i$} \put(58,208){\scriptsize $n$} \put(58,186){\scriptsize $n$}
\put(80,196){$=$}  
\put(141,213){\scriptsize $i$} \put(161,208){\scriptsize $n$} \put(161,186){\scriptsize $n$}
\put(-15,140){$\displaystyle =   \sum_{j,k=0}^{i} \alpha_{i,n}^j \bar{\alpha}_{i,n}^k$}
\put(89,162){\scriptsize $n-j$} \put(128,162){\scriptsize $n$}
\put(69,148){\scriptsize $j$} \put(69,133){\scriptsize $k$}
\put(91,152){\scriptsize $i-j$} \put(91,128){\scriptsize $i-k$}
\put(114,149){\scriptsize $j$} \put(114,130){\scriptsize $k$}
\put(89,118){\scriptsize $n-k$} \put(128,118){\scriptsize $n$}
\put(-15,82){$\displaystyle =   \sum_{0 \leq j \leq k \leq i}  \sum_{l=0}^{j}
\alpha_{i,n}^j \bar{\alpha}_{i,n}^k \beta_{i-k,i-j:i,i}^l$}
\put(169,112){\scriptsize $n-j$} \put(207,112){\scriptsize $n$}
\put(147,102){\scriptsize $j$}  \put(144,83){\scriptsize $j-l$}  \put(147,66){\scriptsize $k$}
\put(176,97){\scriptsize $l$} \put(174,69){\scriptsize $k'_l$} 
\put(194,104){\scriptsize $j$} \put(173,83){\scriptsize $j-l$}  \put(194,62){\scriptsize $k$}
\put(169,54){\scriptsize $n-k$} \put(207,54){\scriptsize $n$}
\put(40,26){$\displaystyle +   \sum_{0 \leq k < j \leq i}  \sum_{l=0}^{k}
\alpha_{i,n}^j \bar{\alpha}_{i,n}^k \beta_{i-j,i-k:i,i}^l$}
\put(240,56){\scriptsize $n-j$}  \put(278,56){\scriptsize $n$}
\put(218,46){\scriptsize $j$}  \put(215,26){\scriptsize $k-l$}  \put(218,10){\scriptsize $k$}
\put(246,41){\scriptsize $j'_l$} \put(246,13){\scriptsize $l$}
\put(265,58){\scriptsize $j$} \put(244,26){\scriptsize $k-l$}  \put(265,6){\scriptsize $k$}
\put(240,-2){\scriptsize $n-k$} \put(278,-2){\scriptsize $n$}
\end{picture}
\centering
\includegraphics[width=100mm,pagebox=cropbox,clip]{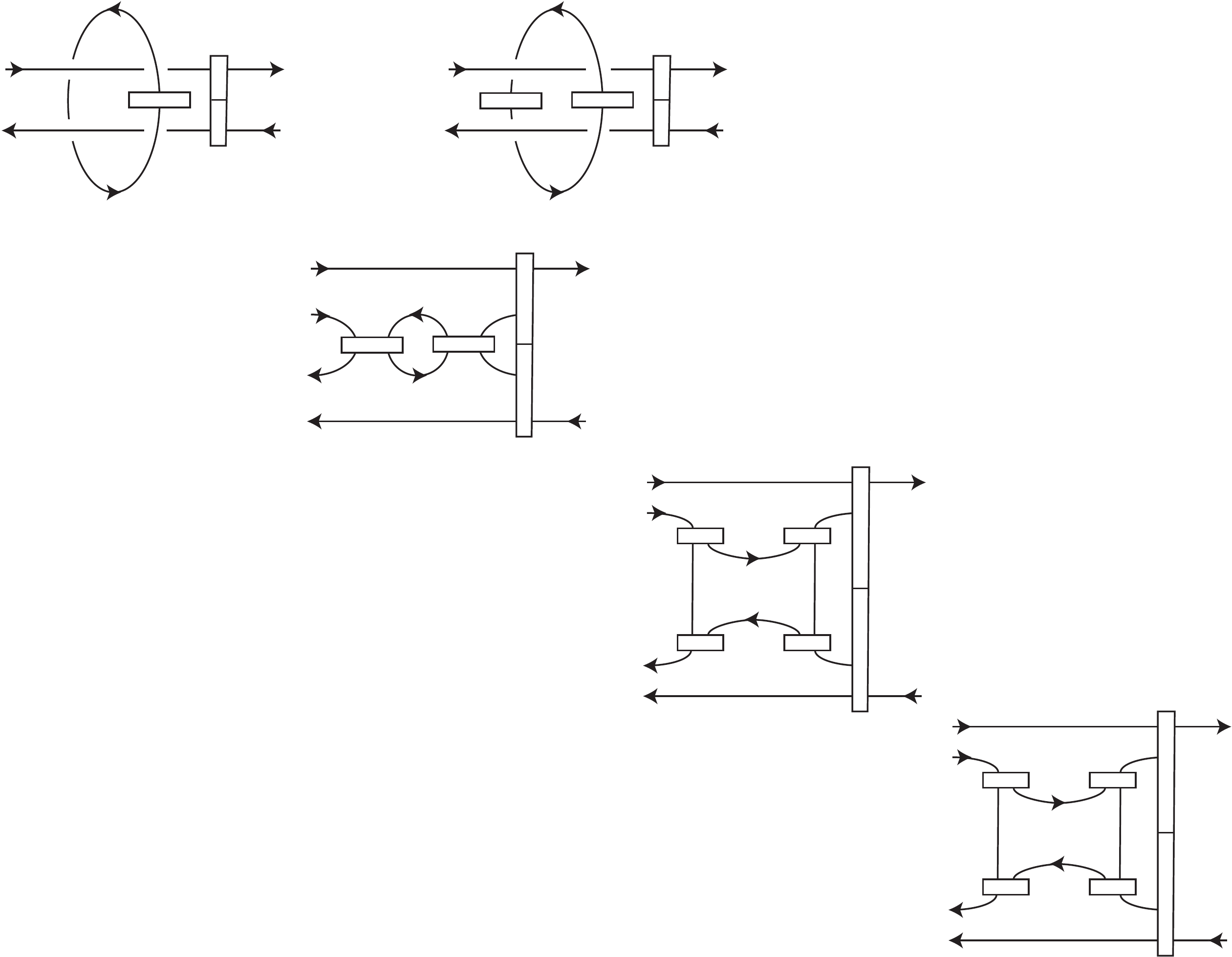}
\caption{$e_{H_i}(D_{n,n})$}
\label{fig:theta0}
\end{figure}
\noindent
where 
$k'_l=k-j+l,  j'_l=j-k+l$.
From the vanishing property of $D_{n,n}$, the first double sum does not vanish 
if $l=j$, and the second double sum does not vanish if $l=k$.
Hence, we obtain
\begin{align*}
\sigma_{n,i}= 
\sum_{0 \leq j \leq k \leq i}  
\alpha_{i,n}^j \bar{\alpha}_{i,n}^k \beta_{i-k,i-j:i,i}^j
+
\sum_{0 \leq k < j \leq i} 
\alpha_{i,n}^j \bar{\alpha}_{i,n}^k \beta_{i-j,i-k:i,i}^k.
\end{align*}
Furthermore, since $\beta$ has the following symmetry:
\begin{align*}
\beta_{i-k,i-j:i,i}^j 
&= \frac{  \{ j \}_j   }{   \{ i \}_{i-k}   }  \frac{  \{ j \}_j   }{  \{ i \}_{i-k}    } \{ i-j\}_{i-j-k}
\left[ \begin{array}{@{\,}c@{\,}}i-k \\ j \end{array} \right]
\{ i -1;a\}_{i-j-k} \\
&=
\frac{ \{ j \}! \{ k \}! }{  \{ i \}_j \{ i \}_k  \{ i-j-k \}!   } \{i-1;a \}_{i-j-k} \\
&=  \beta_{i-j,i-k:i,i}^k,
\end{align*}
$\sigma_{n,i}$ is given by
\begin{align*}
\sigma_{n,i} &= \sum_{\substack{ 0 \leq j+k \leq i \\ j < k }}  
(\alpha_{i,n}^j \bar{\alpha}_{i,n}^k     + \alpha_{i,n}^k \bar{\alpha}_{i,n}^j     ) \beta_{i-k,i-k:i,i}^j 
+ 
\sum_{j=0}^{\lfloor \frac{i}{2} \rfloor } \alpha_{i,n}^j \bar{\alpha}_{i,n}^j  \beta_{i-j,i-j:i,i}^j \\
&=
\sum_{\substack{ 0 \leq j+k \leq i \\ j < k }}  
\Bigl(
(-1)^j
 a^{-j+k} q^{\epsilon_{j,k}} 
+(-1)^k a^{-k+j} q^{\epsilon_{k,j}}
\Bigr) \frac{ \{n\}_j \{ n \}_k     }{  \{ i -j - k \}!  } \{ i-1;a \}_{i-j-k}               \\ 
&\qquad + \sum_{j=0}^{\lfloor \frac{i}{2} \rfloor } 
(-1)^j \frac{ \{n\}_j \{ n \}_j     }{  \{ i -2 j  \}!  } \{ i-1;a \}_{i- 2 j}.
\end{align*}
\end{proof}

\subsection{The eigenvalue for the encircling map of $E_i$}
\begin{lemma}\label{lemma:tau}
For an integer $i \geq 0$,
we have $e_{E_i}(D_{n,n})=\tau_{n,i}D_{n,n}$, where 
\begin{align*}
\tau_{n,i} =  
\frac{   \{ -i +1;a \}_i      }{  \{ i \}!      }    +  \frac{   \{ -i;a \}_{i-1}      }{  \{ i-1 \}!     }  \{ n \} \{ n-1;a \}.    
\end{align*}
\end{lemma}

\begin{proof}
$e_{E_i}(D_{n,n})$ is described by
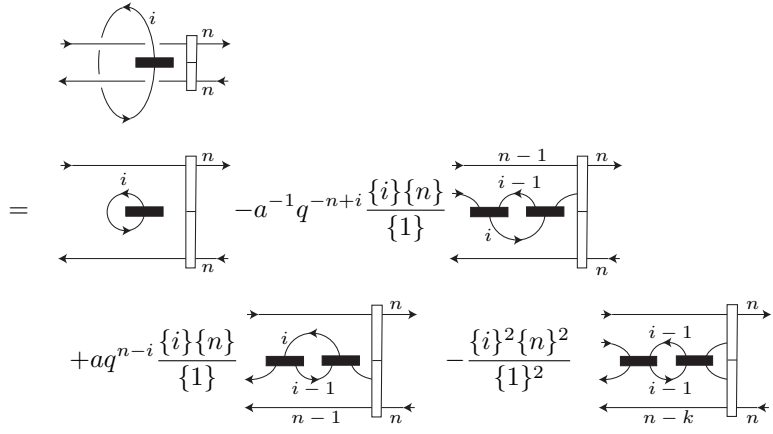
\begin{figure}[H]
\begin{picture}(0,0)(0,0)
\put(38,149){\scriptsize $i$} \put(58,144){\scriptsize $n$} \put(58,122){\scriptsize $n$}
\put(27,88){\scriptsize $i$} \put(58,97){\scriptsize $n$} \put(58,55){\scriptsize $n$}
\put(-15,76){$=$} \put(70,76){$\displaystyle - a^{-1}q^{-n+i}\frac{ \{ i \} \{ n \}  }{ \{ 1 \}   }$}
\put(170,97){\scriptsize $n-1$}  \put(207,97){\scriptsize $n$}
\put(170,87){\scriptsize $i-1$} \put(165,67){\scriptsize $i$} 
\put(207,54){\scriptsize $n$}
\put(8,19){$\displaystyle+ aq^{n-i}\frac{ \{ i \} \{ n \}  }{ \{ 1 \}   }$}
\put(129,41){\scriptsize $n$}
\put(88,28){\scriptsize $i$} \put(92,9){\scriptsize $i-1$} 
\put(92,-2){\scriptsize $n-1$}  \put(129,-2){\scriptsize $n$}
\put(150,19){$\displaystyle -  \frac{ \{ i \}^2 \{ n \}^2  }{ \{ 1 \}^2   }  $}
\put(264,41){\scriptsize $n$}
\put(227,30){\scriptsize $i-1$} \put(227,9){\scriptsize $i-1$}  
\put(225,-2){\scriptsize $n-k$} \put(264,-2){\scriptsize $n$}
\end{picture}
\centering
\includegraphics[width=95mm,pagebox=cropbox,clip]{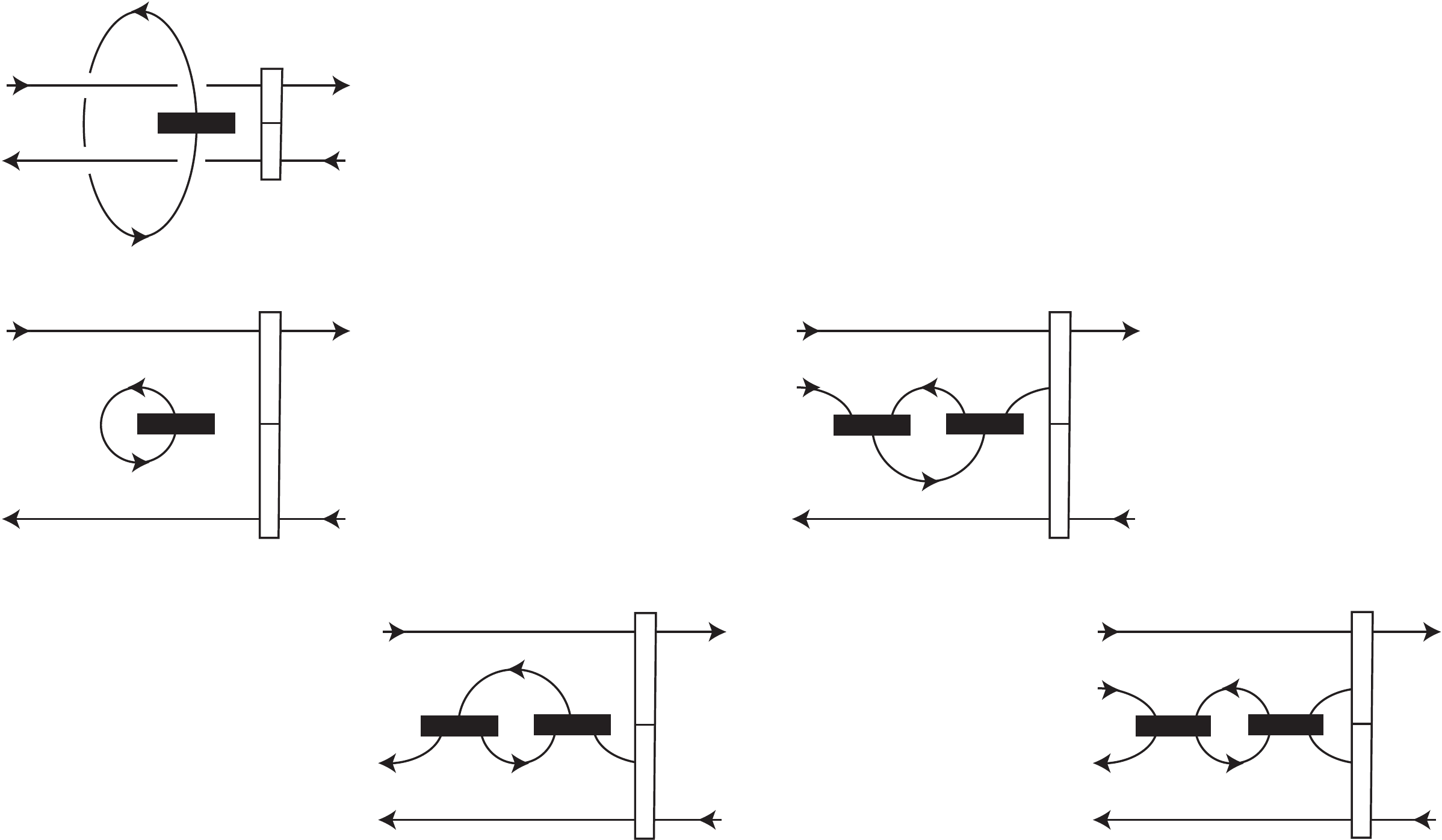}
\caption{$e_{E_i}(D_{n,n})$}
\label{fig:theta1}
\end{figure}
\noindent
Then, the eigenvalue is given by
\begin{align*}
&\frac{  \{ -i+1;a \}_i   }{  \{ i \}!   } +
 ( aq^{n-i} - a^{-1}q^{-n+i}    )\frac{ \{ i \} \{ n \}  }{ \{ 1 \}   }   \frac{ \{ -i +1;a \}_{i-1}   }{  \{i \}_{i-1}       } \\
&\qquad - \frac{ \{ i \}^2 \{ n \}^2  }{ \{ 1 \}^2   }  \beta_{i-1,i-1:i,i}^1 (a,-q^{-1}) \\
=\;
& \frac{   \{ -i +1;a \}_i      }{  \{ i \}!      }    +  \frac{   \{ -i;a \}_{i-1}      }{  \{ i-1 \}!     }  \{ n \} \{ n-1;a \}.    
\end{align*}
\end{proof}
\begin{remark}
For $n=0$, we have the following:
\begin{align*}
\sigma_{0,i} =\langle H_i \rangle = \frac{ \{i-1;a \}_i }{ \{ i \}! }, \qquad
\tau_{0,i} =\langle E_i \rangle  = \frac{ \{-i+1;a \}_i }{ \{ i \}! }.
\end{align*}
\end{remark}

\subsection{The eigenvalue for the encircling map of $R_i$}
In this subsection, for $i=0,\ldots,n$, we define  elements $R_i \in \mathcal{S}$,
which are an important role to calculate the colored HOMFLY-PY polynomial.
Actually, because the behavior of $H_n$ concerning the encircling map is 
very complicated,
we set $R_n$ as a linear combination of $H_i (i=0,\ldots, n)$
so that the behavior of $R_n$ is simple. 
\begin{definition}
For $i=0, \ldots, n$, we recursively define elements $R_i$ as follows:
\begin{align*} 
R_0 &= H_0 = 1, \\
R_1 &= H_1 - \frac{\{ 0;a \}}{\{ 1 \}}, \\
R_2 &= H_2 - \frac{\{ 2;a \}}{\{ 1 \}} R_1 - \frac{\{ 1;a \}_{2}}{\{ 2 \}!}R_0, \\
     &\vdots \\
R_n &= H_n - \sum_{i=0}^{n-1} \frac{ \{ n-1+i;a \}_{n-i} }{ \{ n-i\}! } R_i.
\end{align*}
\end{definition}

\begin{proposition}
For an integer $i \geq 0$,
we have $e_{R_i}(D_{n,n}) = \theta_{n,i} D_{n,n}$, where  $\theta_{n,i} = \{ n \}_{i} \{ n+i-2;a  \}_{i}$.
\end{proposition}

\begin{proof}
By the defition of $R_n$, each $H_i$ is expressed by
\begin{align*}
H_i  =  \sum_{j=0}^{i} \frac{ \{ i-1+j;a \}_{i-j} }{ \{ i-j\}! } R_j.  
\end{align*}
To show the statement, it is enough to prove that the map $e_{H_i}(D_{n,n})$
has the eigenvalue 
\begin{align*}
 \sum_{j=0}^{i} \frac{ \{ i-1+j;a \}_{i-j} }{ \{ i-j\}!} \theta_{n,_j}.
\end{align*}
We will show it by induction.
The statement holds  for $i=0, 1, 2$ by the definition of $H_0$ and Lemma \ref{lemma:sigma}.
To use  the  hypothesis of  induction, 
we express $H_i$ by $H_j$ for $j=0,\ldots,i-1$.
By the determinant formula of Proposition 4 in \cite{Kawagoe1}, we have
\begin{align*}
H_i &= \sum_{j=1}^i (-1)^{j-1} H_{i-j} E_{j} \\
&= H_{i-1} E_1 - H_{i-2} E_2 + H_{i-2} E_2 - \cdots +(-1)^{i-1} E_i.
\end{align*}
We apply the  hypothesis  of induction to $H_j (j=0,\ldots,i-1)$ and Lemma \ref{lemma:tau} to 
$E_j (j=1, \ldots, i)$.
To make the proof easy to read and avoid long equations, we set $\eta_{n,k}$ by 
\begin{align*}
\eta_{n,k} = \{ n - k \} \{ n + k -1;a \}. 
\end{align*}
Then, the eigenvalue induced by  $e_{H_i}$ is given by
\begin{align}
&\Bigl( \sum_{k=0}^{i-1} \frac{ \{ i-2+k;a \}_{i-1-k} }{ \{ i-1-k\}! } \theta_{n,k} \Bigr) 
 \Bigl( \eta_{n,0} + \frac{\{ 0;a \}}{\{ 1 \}}   \Bigr)  \label{eq:scalar1}\\
- 
&\Bigl( \sum_{k=0}^{i-2} \frac{ \{ i-3+k;a \}_{i-2-k} }{ \{ i-2-k\}!  } \theta_{n,k} \Bigr) 
 \Bigl( \frac{\{ -2;a \}}{\{ 1 \}} \eta_{n,0} + \frac{\{ -1;a \}_{2}}{\{ 2 \}! }   \Bigr)  \label{eq:scalar2} \\
+
&\Bigl( \sum_{k=0}^{i-3} \frac{ \{ i-4+k;a \}_{i-3-k} }{ \{ i-3-k\}!  } \theta_{n,k}  \Bigr) 
 \Bigl( \frac{\{ -3;a \}_{2}}{\{ 2 \}! }  \eta_{n,0} + \frac{\{ -2;a \}_{3}}{\{ 3 \}! }   \Bigr) \label{eq:scalar3} \\
& \quad\vdots \notag  \\
+
&(-1)^{i-1} \Bigl( \frac{\{ -i;a \}_{i-1}}{\{ i-1 \}! } \eta_{n,0}+ \frac{\{ -i+1;a \}_{i}}{\{ i \}! }   \Bigr). \notag
\end{align}
To explain more simply, we first demonstrate 
how $\theta_{n,i}, \theta_{n,i-1}.\theta_{n,i-2}$ and their coefficients are obtained.
We remark that $\eta_{n,0}$ satisfies the following identities:
\begin{align*}
\eta_{n,0} = \eta_{n,i-1} + \eta_{i-1,0} =   \eta_{n,i-2} + \eta_{i-2,0} =  \eta_{n,i-3} + \eta_{i-3,0}.  
\end{align*}
We apply these identities to the first three terms  (\ref{eq:scalar1}), (\ref{eq:scalar2}), (\ref{eq:scalar3}).
Then, the eigenvalue is described by
\begin{align*}
&\Biggl(\theta_{n,i-1} 
 \Bigl( \eta_{n,i-1}+ \eta_{i-1,0}  + \frac{\{ 0;a \}}{\{ 1 \}}   \Bigr) \\
&\;\; + \frac{ \{2i-4;a \}  }{ \{1\}   }\theta_{n,i-2} 
\Bigl(  \eta_{n,i-2}  + \eta_{i-2,0}  + \frac{\{ 0;a \}}{\{ 1 \}}   \Bigr) \\
&\;\; + \frac{ \{2i-5;a \}_{2}  }{ \{2\}!   }\theta_{n,i-3} 
\Bigl(  \eta_{n,i-3}  +\eta_{i-3,0}  + \frac{\{ 0;a \}}{\{ 1 \}}   \Bigr) \\
&\;\; + \Bigl( \sum_{k=0}^{i-4} \frac{ \{ i-2+k;a \}_{i-1-k} }{ \{ (i-1)-k\}!} \theta_{n,k} \Bigr) 
 \Bigl(  \eta_{n,0} + \frac{\{ 0;a \}}{\{ 1 \}}   \Bigr) \Biggr)   \\
- 
&\Biggl( \theta_{n,i-2} 
\Bigl( \frac{ \{-2;a \}  }{ \{1\} }  (\eta_{n,i-2}  +\eta_{i-2,0})  + \frac{\{ -1;a \}_{2}}{\{ 2 \}!}   \Bigr)  \\
&\;\; + 
\frac{ \{2i-6;a \}  }{ \{1\}   }\theta_{n,i-3} 
\Bigl( \frac{ \{-2;a \}  }{ \{1\}   }  (  \eta_{n,i-3}  +\eta_{i-3,0}) 
 + \frac{\{ -1;a \}_{2}}{\{ 2 \}!}   \Bigr)   \\
&\;\; +\Bigl( \sum_{k=0}^{i-4} \frac{ \{ i-3+k;a \}_{i-2-k} }{ \{ i-2-k\}!} \theta_{n,k} \Bigr) 
 \Bigl( \frac{\{ -2;a \}}{\{ 1 \}} \eta_{n,0}  + \frac{\{ -1;a \}_{2}}{\{ 2 \}! } \Bigr) \Biggr)  \\
+&\Biggl(\theta_{n,i-3} 
 \Bigl( \frac{ \{-3;a \}_{2}  }{ \{2\}!   }  (  \eta_{n,i-3}  +\eta_{i-3,0})  + \frac{\{ -2;a \}_{3}}{\{ 3 \}!}   \Bigr)  \\
&\;\;+
\Bigl( \sum_{k=0}^{i-4} \frac{ \{ i-4+k;a \}_{i-3-k} }{ \{ i-3-k\}!} \theta_{n,k}  \Bigr) 
\Bigl( \frac{\{ -3;a \}_{2}}{\{ 2 \}! }   \eta_{n,0} + \frac{\{ -2;a \}_{3}}{\{ 3 \}! }   \Bigr) \Biggr)+ \cdots. 
\end{align*}
Using the following identity:
\begin{align*}
\theta_{n,k-1} \eta_{n,k-1}  = \theta_{n,k} \quad   (k=1, \ldots),
\end{align*}
the eigenvalue with respect to $\theta_{n,i}, \theta_{n,i-1}, \theta_{n,i-2}$  is expressed by
\begin{align*}
&\theta_{n,i} + 
\Bigl(
 \eta_{i-1,0}  + \frac{\{ 0;a \}}{\{ 1 \}} + \frac{ \{2i-4; a \}  }{ \{1\}   }
-\frac{ \{-2;a \}  }{ \{1\}   } 
 \Bigr) 
\theta_{n,i-1} \\
&+
\Bigl(
 \frac{ \{2i-4; a \}  }{ \{1\}}  (     \eta_{i-2,0}  + \frac{\{ 0;a \}}{\{ 1 \}}            )
+ \frac{ \{2i-5;a \}_{2}  }{ \{2\}!   } -   \frac{ \{-2;a \}  }{ \{1\}   }  \eta_{i-2,0} \\
& - \frac{\{ -1;a \}_{2}}{\{ 2 \}!}  -    \frac{ \{2i-6;a \}  }{ \{1\}   } \frac{ \{-2;a \}  }{ \{1\}   }    
+ \frac{ \{-3;a \}_{2}  }{ \{2\}!   }  \Bigr) \theta_{n,i-2} + \cdots \\
=\;&\theta_{n,i} + \frac{\{ 2i-2;a \}}{\{ 1 \}} \theta_{n,i-1}    
+\frac{\{ 2i-3;a \}_{2}}{\{ 2 \}!} \theta_{n,i-2} + \cdots.
\end{align*}

Next, we consider a general situation. 
The first term  (\ref{eq:scalar1}) generates 
$\theta_{n,i}, \theta_{n,i-1}, \ldots, \theta_{n,0}$.
The second term  (\ref{eq:scalar2}) generates $\theta_{n,i-1}, \theta_{n,i-2}, \ldots, \theta_{n,0}$,
and the third term  (\ref{eq:scalar3}) generates $\theta_{n,i-2}, \theta_{n,i-3}, \ldots, \theta_{n,0}$.
Furthermore, this demonstration implies that 
the $j$-th term generates  $\theta_{n,i-j+1}, \theta_{n,i-j}, \ldots, \theta_{n,0}$.
Therefore,  $\theta_{n,i-j+1}$ comes from
the first, the second, ..., the $j$-th terms. 
By applying the identity
\begin{align*}
\eta_{n,0} = \eta_{n,i} + \eta_{i,0} \quad (i=1,\ldots),
\end{align*}
we express  the $l$-th term  as follows: 
\begin{align*}
&(-1)^{l-1}
\Bigl( \sum_{k=0}^{i-l} \frac{ \{ i-l-1+k;a \}_{i-l-k} }{ \{ i-l-k\}! } \theta_{n,k}  \Bigr)  
\Bigl( \frac{\{ -l;a \}_{l-1}}{\{ l-1 \}! }  \eta_{n,0}+ \frac{\{ -l+1;a \}_{l}}{\{ l \}! }   \Bigr)  \\
= 
&(-1)^{l-1}  \sum_{k=0}^{i-l} 
\Bigl(
\frac{ \{ i-l-1+k;a \}_{i-l-k} }{ \{ i-l-k\}!  } \theta_{n,k}
\Bigr) 
\Biggl(
 \frac{\{ -l;a \}_{l-1}}{\{ l-1 \}! } \Bigl( \eta_{n,k} + \eta_{k,0}    \Bigr)
+ \frac{\{ -l+1;a \}_{l}}{\{ l \}! } 
\Biggr) \\
=
&(-1)^{l-1}  \sum_{k=0}^{i-l} 
\Biggl(
\frac{ \{ i-l-1+k;a \}_{i-l-k} }{ \{ i-l-k\}!  } \frac{\{ -l;a \}_{l-1}}{\{ l-1 \}! }  \theta_{n,k+1} \\
&\quad
+\frac{ \{ i-l-1+k;a \}_{i-l-k} }{ \{ i-l-k\}!  }
\Bigl(
 \frac{\{ -l;a \}_{l-1}}{\{ l-1 \}! }  \eta_{k,0}  +   \frac{\{ -l+1;a \}_{l}}{\{ l \}! } 
\Bigr)
 \theta_{n,k}
\Biggr)
\end{align*}
We consider  the contribution of the coefficient of $\theta_{n, i-j+1}$ for $l=1,\ldots,j$.
In the upper term including $\theta_{n,k+1}$, the contribution yields when $k=i-j$.
In the lower term including $\theta_{n,k}$,   the contribution yields when $k=i-j+1$.
Hence, the coefficient of $\theta_{n, i-j+1}$ is described by
\begin{align*}
&\sum_{l=1}^{j} (-1)^{l-1} \frac{  \{2i-j-l-1;a\}_{j-l}  }{  \{j-l\}!  } 
\frac{  \{-l;a\}_{l-1}  }{  \{l-1\}!  } \\
+&\sum_{l=1}^{j-1} (-1)^{l-1} \frac{  \{2i-j-l;a\}_{j-l-1}  }{  \{j-l-1\}!  }
\Bigl(  \frac{  \{-l;a\}_{l-1}  }{  \{l-1\}!  }  \eta_{i-j+1,0} 
+  \frac{  \{ -l+1;a \}_{l}  }{  \{ l \}!  } \Bigr).  
\end{align*}
%
Furthermore, we rearrange the following equation: 
\begin{align*}
&\frac{  \{ 2i - j -l -1;a\}_{j-l}  }{ \{ j-l  \}!   }  \frac{  \{ -l;a\}_{l-1}  }{ \{ l-1 \}!   } \\
+\;
&\frac{  \{ 2i - j -l ;a\}_{j-l-1}  }{ \{ j-l-1 \}!   } 
\Bigl(
 \frac{  \{ -l;a\}_{l-1}  }{ \{ l-1 \}!} \eta_{i-j+1,0}  +  \frac{  \{ -l+1;a\}_{l}  }{ \{ l \}!} 
\Bigr) 
\end{align*}
Using the identity
\begin{align*}
\eta_{n,0} = \{ n \} \{ n-1;a \} = \frac{  \{ 2n;a \}  }{  \{1 \}  }  -  \frac{  \{ 2(n-1);a \}  }{  \{1 \}  }
+  \frac{  \{ -2;a \}  }{  \{1 \}  } - \frac{  \{ 0;a \}  }{  \{1 \}  },
\end{align*}
it is expressed by
\begin{align*}
&\frac{  \{ 2i - j -l ;a\}_{j-l-1}  }{ \{ j-l-1 \}!   }  \frac{  \{ -l;a\}_{l-1}  }{ \{ l-1 \}!} \\
&\qquad \times \Bigl( \frac{ \{ 2(i- j+1);a \} }{ \{ 1 \} } -  \frac{ \{ 2( i - j);a \} }{ \{ 1 \} }
+ \frac{ \{ -2;a \} }{ \{ 1 \} } -  \frac{ \{ 0;a \} }{ \{ 1 \} }  \Bigr) \\
&\quad 
+ 
\frac{  \{ 2i - j -l -1;a\}_{j-l}  }{ \{ j-l  \}!   }  \frac{  \{ -l;a\}_{l-1}  }{ \{ l-1 \}!   } 
+\frac{  \{ 2i - j -l ;a\}_{j-l-1}  }{ \{ j-l-1 \}!   }  \frac{  \{ -l+1;a\}_{l}  }{ \{ l \}!} \\
= &\Biggl( \frac{  \{ 2i - j -l ;a\}_{j-l-1}  }{ \{ j-l-1 \}!   }  \frac{  \{ -l;a\}_{l-1}  }{ \{ l-1 \}!}
 \Bigl( \frac{ \{ 2(i- j+1);a \} }{ \{ 1 \} } -  \frac{ \{ 2( i - j);a \} }{ \{ 1 \} } \Bigr) \\
&\quad +\frac{  \{ 2i - j -l -1;a\}_{j-l}  }{ \{ j-l \}!   }  \frac{  \{ -l;a\}_{l-1}  }{ \{ l-1 \}!}  \Biggr) \\
& +\Biggl(\frac{  \{ 2i - j -l ;a\}_{j-l-1}  }{ \{ j-l-1 \}!   }  \frac{  \{ -l;a\}_{l-1}  }{ \{ l-1 \}!} 
\Bigl( \frac{ \{ -2;a \} }{ \{ 1 \} } -  \frac{ \{ 0;a \} }{ \{ 1 \} }  \Bigr)  \\ 
&\qquad + \frac{  \{ 2i - j -l ;a\}_{j-l-1}  }{ \{ j-l-1 \}!   }  \frac{  \{ -l+1;a\}_{l}  }{ \{ l \}!} \Biggr) \\
%
=  &\frac{  \{ 2i - j -l -1;a\}_{j-l-2}  }{ \{ j-l \}!   }  \frac{  \{ -l;a\}_{l-1}  }{ \{ l-1 \}!} \\
&\quad \times \Biggl( 
\{ j -l \} \{ 2i-j-l;a \} \Bigl( \frac{\{ 2(i -j +1); a \} }{ \{ 1 \} } 
-  \frac{  \{ 2(i -j);a \} }{  \{ 1 \}  } \Bigr) \\
&\qquad + \{ 2(i-j)+1;a \} \{ 2(i -j);a \} 
\Biggr) \\
&+ \frac{  \{ 2i - j -l ;a\}_{j-l-1}  }{ \{ j-l-1 \}!   }  \frac{  \{ -l+1;a\}_{l-2}  }{ \{ l \}!} \\
&\quad \times \Biggl( \{ l \} \{ -l;a \} \Bigl(\frac{ \{ -2;a \} }{ \{ 1 \} } -  \frac{ \{ 0;a \} }{ \{ 1 \} }\Bigr)   
 + \{  -1;a\} \{ 0;a\}   \Biggr) \\
= \; &\frac{  \{ 2i - j -l -1;a\}_{j-l-2}  }{ \{ j-l \}!   }  \frac{  \{ -l;a\}_{l-1}  }{ \{ l-1 \}!} \\
&\quad \times \Bigl( 
\{ 2i-j-l+1;a \} \{ 2i-j-l;a \} +\{ j -l \} \{ j-l-1 \} \Bigr) \\
&+ \frac{  \{ 2i - j -l ;a\}_{j-l-1}  }{ \{ j-l-1 \}!   }  \frac{  \{ -l+1;a\}_{l-2}  }{ \{ l \}!} 
 \Bigl( \{ -l-1;a \} \{ -l;a \} +   \{ l \}  \{ l-1 \} \Bigr).
\end{align*}
Finally, it is given by
\begin{align*}
&\frac{  \{ 2i - j -l +1;a\}_{j-l}  }{ \{ j-l \}!   }  \frac{  \{ -l;a\}_{l-1}  }{ \{ l-1 \}!   }
+ \frac{  \{ 2i - j -l -1;a\}_{j-l-2}  }{ \{ j-l -2 \}!   }  \frac{  \{ -l;a\}_{l-1}  }{ \{ l-1 \}!   } \\
&\quad +\frac{  \{ 2i - j -l ;a\}_{j-l-1}  }{ \{ j-l-1 \}!   }  \frac{  \{ -l-1;a\}_{l}  }{ \{ l \}!   }
+\frac{  \{ 2i - j -l;a\}_{j-l-1}  }{ \{ j-l-1 \}!   }  \frac{  \{ -l+1;a\}_{l-2}  }{ \{ l-2 \}!   }.
\end{align*}
Here, we consider the term $\{ \quad; a \}_{i}$ as $0$ when $i$ is negative.    
Then, the above expression is valid for $l=1,\ldots,j$.
We sum up these terms from $l=1$ to $l=j$ with multiplication by $(-1)^{l-1}$. 
It is equal to
\begin{align*}
\frac{  \{ 2i - j ;a\}_{j-1}  }{ \{ j-1 \}!   }.
\end{align*}
This is  the coefficient of $\theta_{n, i-j+1}$. 
This completes the proof.
\end{proof}
Since $\{ n \}_i =0$ for $i >n$,  we have the following corollary.
\begin{corollary}\label{cor:vanish}
For an integer $i > n$, we have $e_{R_i}(D_{n,n})=0$.
\end{corollary}

\subsection{The eigenvalue  for the encircling map of  $\omega_n $}
In the last subsection, we discuss a relation 
between the encircling map $e$ and the twist maps $t$.
That is, for any $x \in \mathcal{H}_{n,n}$, 
we would like to find an element $\omega_n=\omega_n^+ \in \mathcal{S}$
such that $e_{\omega_n}(x) = t(x)$ for the positive twist map $t$.
In fact, $\omega_n$ is given by the following definition.
\begin{definition}
For an integer $n \geq 0$, we define  $\omega_n \in \mathcal{S}$ by
\begin{align*}
\omega_n = \sum_{i=0}^{n} t_{i} R_i,
\quad
{\text{where}} \;
t_i = \frac{a^i q^{\frac{i(i-1)}{2}}}{ \{ i \}!}.
\end{align*}
\end{definition}
Since $\{ D_{i,i} \; | \; i = 0, \ldots, n \}$ is a basis of $\mathcal{H}_{n,n}$
and $t(D_{i,i}) = a^{2i}q^{2i(i-1)}D_{i,i}$,
if we would like to show  $e_{\omega_n}(x) = t(x)$,  it is enough to show that
\begin{align*}
e_{\omega_{n}}(D_{i,i}) = a^{2i}q^{2i(i-1)}D_{i,i}, \quad (i=0,\ldots,n).
\end{align*}
However, by Corollary \ref{cor:vanish}, it is sufficient to show that
\begin{align*}
e_{\omega_{i}}(D_{i,i}) = a^{2i}q^{2i(i-1)}D_{i,i}, \quad(i=0,1,\ldots).
\end{align*}
Since the map $e_{\omega_n}$ induces the eigenvalue
\begin{align*}
\sum_{i=0}^n \frac{a^i q^{\frac{i(i-1)}{2}}}{ \{ i \}!} \theta_{n,i} 
= \sum_{i=0}^n a^{i} q^{\frac{i(i-1)}{2}} 
\left[ \begin{array}{@{\,}c@{\,}}n \\ i \end{array} \right] 
\{n+i  -2 ; a \}_{i},
\end{align*}
we will prove the following statement.
\begin{proposition}
The following identity holds:
\begin{align*} 
a^{2n}q^{2n(n-1)} 
= & \sum_{i=0}^n a^{i} q^{\frac{i(i-1)}{2}} 
\left[ \begin{array}{@{\,}c@{\,}}n \\ i \end{array} \right] 
\{n+i-2; a \}_{i}.  
\end{align*}
\end{proposition}
\begin{proof}
We consider a triangular array like Pascal's triangle.
The $(0,0)$ entry is $a^{2n}q^{2n(n-1)}$.
First, we split $a^{2n}q^{2n(n-1)} $ into 
\begin{align*}
a^{2n}q^{2n(n-1)} 
& = a^{2n}q^{2n(n-1)} -a^{2n-2}q^{2(n-1)(n-1)}  +a^{2n-2}q^{2(n-1)(n-1)} \\ 
& = a^{2n-1}q^{(2n-1)(n-1)} \{ n-1; a\}       + a^{2n-2}q^{2(n-1)(n-1)}. 
\end{align*}
The first term of the last equation is the $(1,0)$ entry and the second term of it 
is the $(1,1)$ entry in the triangle.
Next, we split the $(1,0)$ and the $(0,1)$ entries. The $(1,0)$ entry splits into
\begin{align*}
&a^{2n-1}q^{(2n-1)(n-1)} \{ n-1; a\} \\
= \; &(a^{2n-1}q^{(2n-1)(n-1)} -  a^{2n-3}q^{2n(n-1) -3n +1}) \{ n-1; a\} \\
&\quad 
+  a^{2n-3}q^{2n(n-1) -3n +1} \{ n-1; a\} \\
= \; &a^{2n-2}q^{2n(n-1)-2n+1} \{ n-1; a\}  \{ n; a\} +   a^{2n-3}q^{2n(n-2)} \{ n-1; a\},   
\end{align*}
and  the $(1,1)$ entry splits into
\begin{align*}
&a^{2n-2}q^{2(n-1)(n-1)} \\
= \; &a^{2n-2}q^{2(n-1)(n-1)} -  a^{2n-4}q^{2(n-1)(n-2)} +  a^{2n-4}q^{2(n-1)(n-2)}\\
= \; &a^{2n-3}q^{(2n-3)(n-1)} \{ n-1; a\} + a^{2n-4}q^{2(n-1)(n-2)}.  
\end{align*}
Then, we have the following split as the second row:
\begin{align*}
a^{2n}q^{2n(n-1)} 
& = a^{2n-2}q^{2n(n-1)-2n+1} \{ n; a\}_2 \\
&\quad +a^{2n-3}q^{2n(n-1)-3n+2} \left[ \begin{array}{@{\,}c@{\,}}2 \\ 1 \end{array} \right] \{ n-1; a\} 
 + a^{2n-4}q^{2(n-1)(n-2)}.  
\end{align*}
The first, the second, the third term of the right-hand equation are the $(2,0), (2,1), (2,2)$ entries in the triagle.
Let $\lambda_{i,j}$ be  
\begin{align*}
\lambda_{i,j} &= 
a^{2n-i-j}q^{\kappa_{i,j}}
 \left[ \begin{array}{@{\,}c@{\,}}i \\ j \end{array} \right] \{ n+i-j-2; a\}_{i-j}, \\
& \text{where $\displaystyle \kappa_{i,j}=2n(n-1) -n(i+j) - \frac{i(i-3)}{2} + \frac{j(j+1)}{2}$.}
\end{align*}
We show that the $(i,j)$ entry in the triangle is given by $\lambda_{i,j}$ if we demonstrate a suitable split.
We assume that the $(i,j)$ entry in the triangle is expressed by $\lambda_{i,j}$.
Then, we split $\lambda_{i,j}$ into
\begin{align*}
\lambda_{i,j}  &= \lambda_{i,j} -a^{-2}q^{-2(n+i-j-1)}\lambda_{i,j} + a^{-2}q^{-2(n+i-j-1)}\lambda_{i,j} \\
&= a^{-1}q^{-(n+i-j-1)} \lambda_{i,j}  \{ n+i-j-1; a\} +  a^{-2}q^{-2(n+i-j-1)}\lambda_{i,j}.
\end{align*}
We remark that this split agrees with the first row $(\lambda_{1,0}, \lambda_{1,1})$, and
the second row $(\lambda_{2,0}, \lambda_{2,1}, \lambda_{2,2})$.
This split means that the first term moves into the $(i+1,j)$ entry and 
the second term moves into $(i+1,j+1)$ entry.
Here, we calculate the $(i+1,j+1)$ entry  derived from the $(i,j)$ and the  $(i,j+1)$ entries.
It is expressed by
\begin{align*}
& a^{-2}q^{-2(n+i-j-1)}\lambda_{i,j} +  a^{-1}q^{-(n+i-(j+1)-1)} \lambda_{i,j+1}  \{ n+i-(j+1)-1; a\} \\
=& a^{2n-i-j-2}q^{\kappa_{i,j}-2(n+i-j-1)}  \left[ \begin{array}{@{\,}c@{\,}}i \\ j \end{array} \right] \{ n+i-j-2; a\}_{i-j} \\
 & + a^{2n-i-j-2}q^{\kappa_{i,j+1}-2(n+i-j-2)}  \left[ \begin{array}{@{\,}c@{\,}}i \\ j+1 \end{array} \right] 
\{ n+i-j-3; a\}_{i-j-1} \{ n+i-j-2; a\} \\
=& a^{2n-i-j-2}\frac{[i]!}{[j+1]![i-j]!}  \{ n+i-j-2; a\}_{i-j} \\
&\quad \times \Bigl( q^{\kappa_{i,j} -2(n+i-j-1)}[j+1] +q^{\kappa_{i,j+1}-(n+i-j-2)}[i-j]\Bigr) \\
=& a^{2n-i-j-2}\frac{[i]!}{[j+1]![i-j]!}  \{ n+i-j-2; a\}_{i-j} \times q^{\kappa_{i+1,j+1}}[i+1].
\end{align*}
This agrees with $\lambda_{i+1,j+1}$. 
Therefore, when we take the sum of the $i$-th row, and we obtain
\begin{align*}
a^{2n}q^{2n(n-1)} 
= & \sum_{k=0}^i\lambda_{i,k}.
\end{align*}
Finally, put $i=n$. The assertion is obtained. 
\end{proof}
For the negative twist map $t^{-1}$, we  consider the mirror image of elements of the skeins module.
That is, we define $\omega_n^- \in \mathcal{S}$ by
\begin{align*}
\omega_n^- = \sum_{i=0}^{n} \bar{t}_{i} R_i,
\quad
{\text{where}} \;
\bar{t}_i = (-1)^i \frac{a^{-i} q^{\frac{-i(i-1)}{2}}}{ \{ i \}!}.
\end{align*}
Then, we have $e_{\omega_n^-}(D_{i,i}) = t^{-1}(D_{i,i})$ for $i=0,\ldots,n$.
Summarizing the above, we have the following proposition.
\begin{proposition}\label{proposition:omega}
For any $x \in \mathcal{H}_{n,n}$, we have $e_{\omega_n^+}(x)= t(x)$ and 
$e_{\omega_n^-}(x)= t^{-1}(x)$. 
\end{proposition}

\section{A twisting formula}
In this section, we would like to find a twisting formula of $p$ full twists.
As a preparation, we need the following lemma.
\begin{lemma}\label{lemma:ymn} The following holds:
\begin{figure}[H]
\begin{picture}(0,0)(0,0)
\put(6,27){\scriptsize $m$} \put(6,5){\scriptsize $n$}
\put(41,15){$\displaystyle = \sum_{i=0}^{\min\{m,n\}} y_{m,n}^i$} 
\put(128,29){\scriptsize $m-i$}
\put(126,16){\scriptsize $i$}
\put(151,16){\scriptsize $i$} 
\put(130,2){\scriptsize $n-i$} 
\end{picture}
\centering
\includegraphics[width=60mm,pagebox=cropbox,clip]{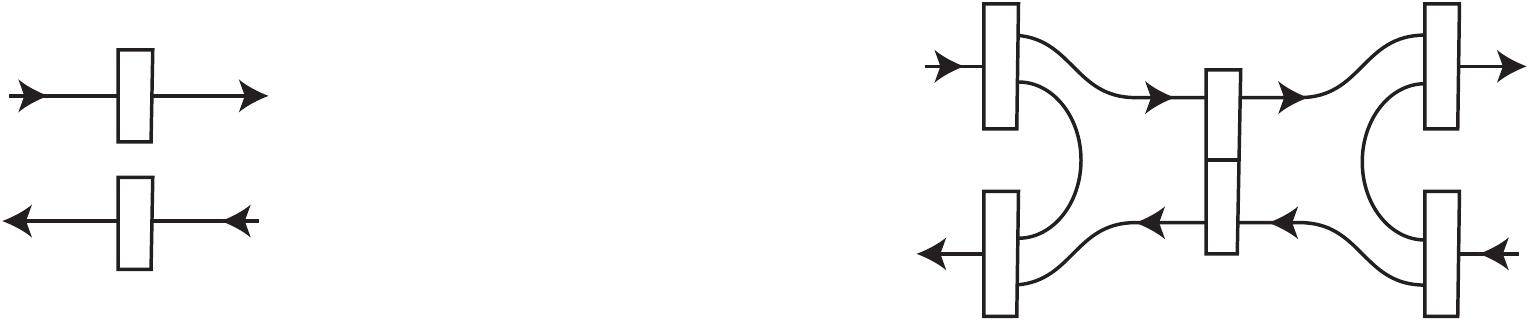}
\label{fig:Hmn}
\end{figure}
\noindent
where $y_{m,n}^i$ is defined by
\begin{align*}
y^i_{m,n} = \frac{  \{m\}_i \{n\}_i  }{ \{ i \}! \{m+n-i-1;a \}_i   }.
\end{align*}
\end{lemma}
\begin{proof}
The definition of the $(m,n)$-th $q$-symmetrizer implies that
the left-hand side is described by a linear combination of 
the $(m-i,n-i)$-th $q$-symmetrizers 
with some coefficients $y_i$ for $i=0,\ldots, \min\{m,n\}$.
We select the $j$-th term in the sum, 
which includes the $(m-j,n-j)$-th $q$-symmetrizer. 
Then, we connect the $j$-th term to both sides.
From Lemma \ref{lemma:beta} and the idempotent property, 
we obtain the coefficient of the $j$-th term, $y_j$,
which is equal to $y_{m,n}^j$ above.  
\end{proof}

We set the elemnet  $\omega_n^p \in \mathcal{S}$ by
\begin{align*}
\omega_n^p = \sum_{i=0}^n t_{i,p} R_i,
\end{align*}
for some $t_{i,p} \in \mathbb{C}$, and we would like to determine 
coefficients $t_{i,p}$
so that $e_{\omega_n^p}(x) = t^p(x)$ for $x \in \mathcal{D}_{n,n}$.

By the idempotent and vanishing properties of $D_{n,n}$,
for some $T_{n,p} \in \mathbb{C}$,
the left-hand side of Figure \ref{fig:twist} is transformed into 
$D_{n,n}$ with the multiplication of $T_{n,p}$.
\begin{figure}[H]
\begin{picture}(0,0)(0,0)
\put(37,81){\scriptsize $n$} \put(51,81){\scriptsize $n$}
\put(25,37){\scriptsize $2p$ crossings}
\put(37,-4){\scriptsize $n$} \put(51,-4){\scriptsize $n$} 
\put(105,37){$\displaystyle =\quad T_{n,p}$}
\put(153,55){\scriptsize $n$} \put(168,55){\scriptsize $n$}
\put(153,21){\scriptsize $n$} \put(168,21){\scriptsize $n$} 
\end{picture}
\centering
\includegraphics[width=60mm,pagebox=cropbox,clip]{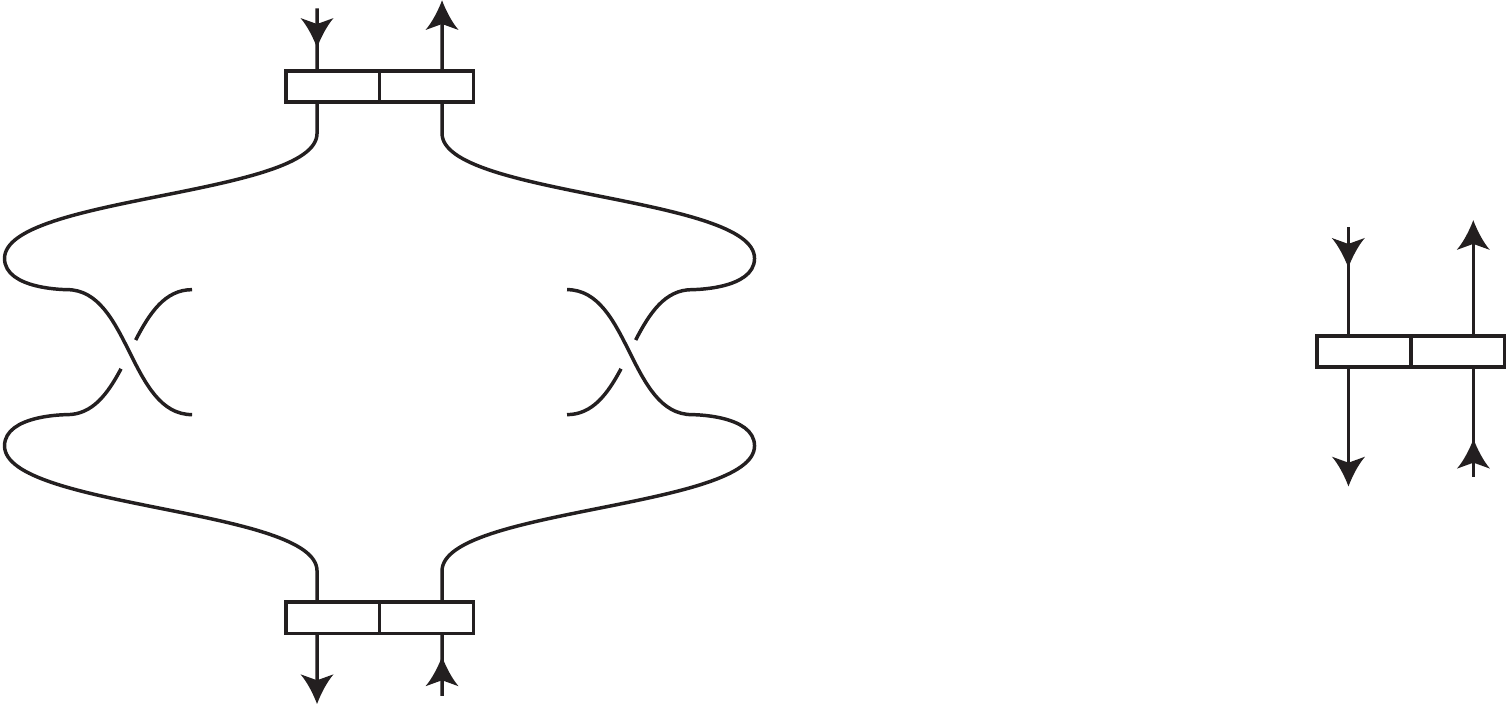}
\caption{}
\label{fig:twist}
\end{figure}
\noindent
To determine $t_{n,p}$, we calculate 
$T_{n,p}$ in two ways.
One is to make use of the $(n-i,n-i)$-th $q$-symmetrizer ($i=0,\ldots,n$),
and the other is to encircle $p$ full twists by $\omega_{n}^{p}$.
These two ways are due to \cite{Masbaum}.

First, using Lemma \ref{lemma:ymn},
we replace $2n$ arcs by  the $(n-i,n-i)$-th $q$-symmetrizer ($i=0,\ldots,n$). 
Next, we undo the $i$ arcs. 
This procedure is shown in Figure \ref{fig:twist0}.
\begin{figure}[H]
\begin{picture}(0,0)(0,0)
\put(37,168){\scriptsize $n$} \put(51,168){\scriptsize $n$}
\put(25,125){\scriptsize $2p$ crossings}
\put(37,84){\scriptsize $n$} \put(51,84){\scriptsize $n$} 
\put(93,125){$\displaystyle = \sum_{i=0}^n y_{n,n}^i$} 
\put(153,146){\scriptsize $n$}
\put(153,106){\scriptsize $n$}
\put(203,130){\scriptsize $i$}
\put(203,121){\scriptsize $i$}
\put(231,125){\scriptsize $i$} 
%
\put(10,37){$\displaystyle = \sum_{i=0}^n y_{n,n}^i (a^{-i}q^{-i(i-1)-2i(n-i)})^{2p}$}
\put(180,38){\scriptsize $i$}
\put(200,47){\scriptsize $n-i$}  \put(240,48){\scriptsize $n-i$}
\put(200,27){\scriptsize $n-i$}  \put(240,27){\scriptsize $n-i$}
\put(275,38){\scriptsize $i$} 
\end{picture}
\centering
\includegraphics[width=95mm,pagebox=cropbox,clip]{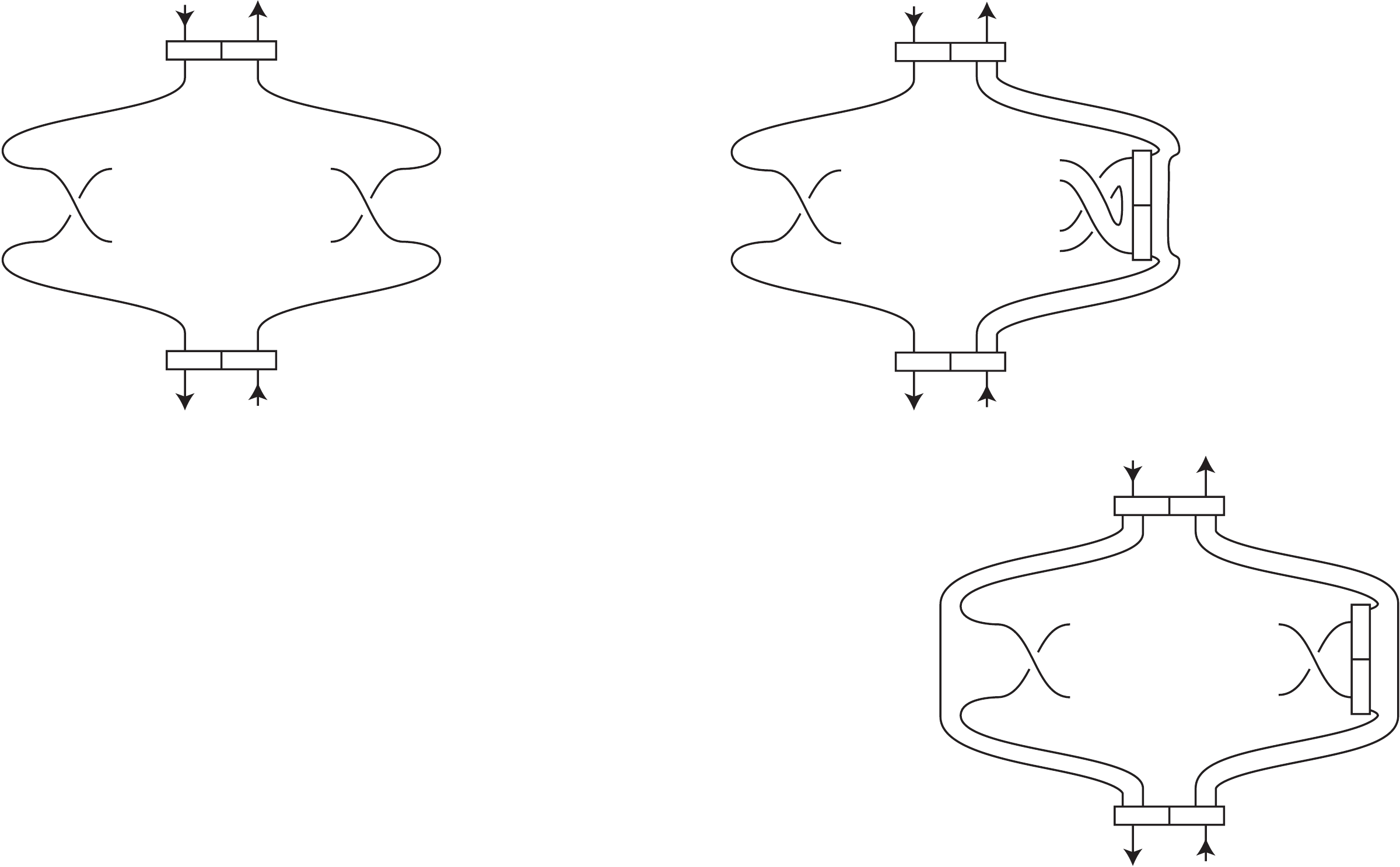}
\caption{}
\label{fig:twist0}
\end{figure}
\noindent
Because of the HOMFLY-PT skein relation and the vanishing property of 
the $(n-i,n-i)$-th $q$-symmetrizer, all crossings are canceled shown in Figure \ref{fig:twist01}.
\begin{figure}[H]
\begin{picture}(0,0)(0,0)
\put(101,37){$=$}
\put(150,37){$\displaystyle = \sum_{j=0}^{n-i} x_{n-i,n-i}^j$} 
\put(234,38){\scriptsize $j$}\put(246,38){\scriptsize $j$}
\end{picture}
\centering
\includegraphics[width=90mm,pagebox=cropbox,clip]{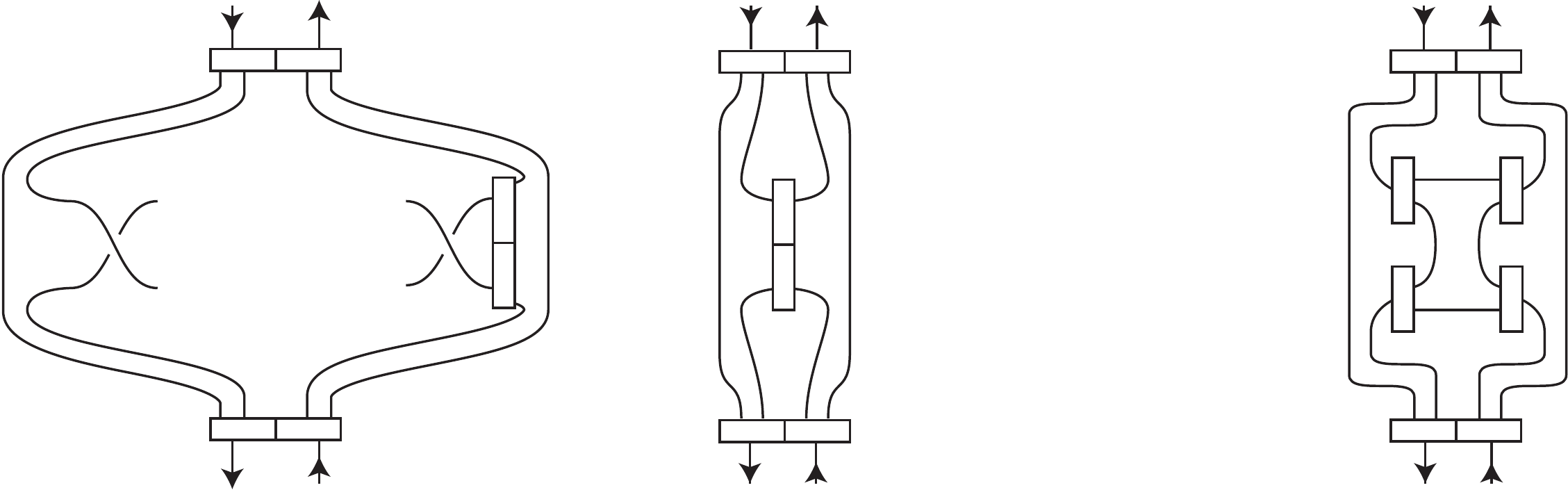}
\caption{}
\label{fig:twist01}
\end{figure}
\noindent
From the vanishing property of the $(n,n)$-th $q$-symmetrizer,
Figure \ref{fig:twist01} does not vanish if $j=n-i$.
Hence we have
\begin{align*}
T_{n,p}
&= \sum_{i=0}^n y_{n,n}^i (a^{-i}q^{-i(i-1)-2i(n-i)})^{2p} x_{n-i,n-i}^{n-i} \\
&= 
 \sum_{i=0}^n  (a^{-i}q^{-i(i-1)-2i(n-i)})^{2p}
\frac{  \{n\}_i \{ n \}_i  }{  \{ i \}! \{2n-i-1;a\}_i  }(-1)^{n-i} \frac{ \{n-i\}!   }{ \{ 2n-2i-2;a\}_{n-i}    } \\
&=
(a^{-n}q^{-n(n-1)})^{2p} \{ n \}!
\sum_{i=0}^n (-1)^i (a^{i}q^{i(i-1)})^{2p}\frac{  \{n\}_{n-i}  }{ \{n-i\}! }
\frac{  \{2i-1;a \}  }{  \{ n+i-1;a \}_{n+1}   } 
\end{align*}

Second, we encircle  $2p$ full twists arcs by  $\omega_{n}^{p}$.
Here, we consider the number of arcs of $\omega_n^p$ at the horizontal line.
From the vanishing property of the $(n,n)$-th $q$-symmetrizer,
The non-vanishing term has $2n$ arcs, which is $R_n$. 
Therefore, we have the transformation shown in Figure \ref{fig:twist1}.
\begin{figure}[H]
\begin{picture}(0,0)(0,0)
\put(25,38){\scriptsize $2p$ crossings}
\put(37,81){\scriptsize $n$} \put(51,81){\scriptsize $n$}
\put(37,-4){\scriptsize $n$} \put(51,-4){\scriptsize $n$}
\put(90,38){$= (a^nq^{n(n-1)})^{-2p}$} 
\put(182,55){$\omega_n^p$}
\put(212,38){$= (a^nq^{n(n-1)})^{-2p}t_{n,p}$}
\put(313,54){ $R_n$} 
\end{picture}
\centering
\includegraphics[width=117.5mm,pagebox=cropbox,clip]{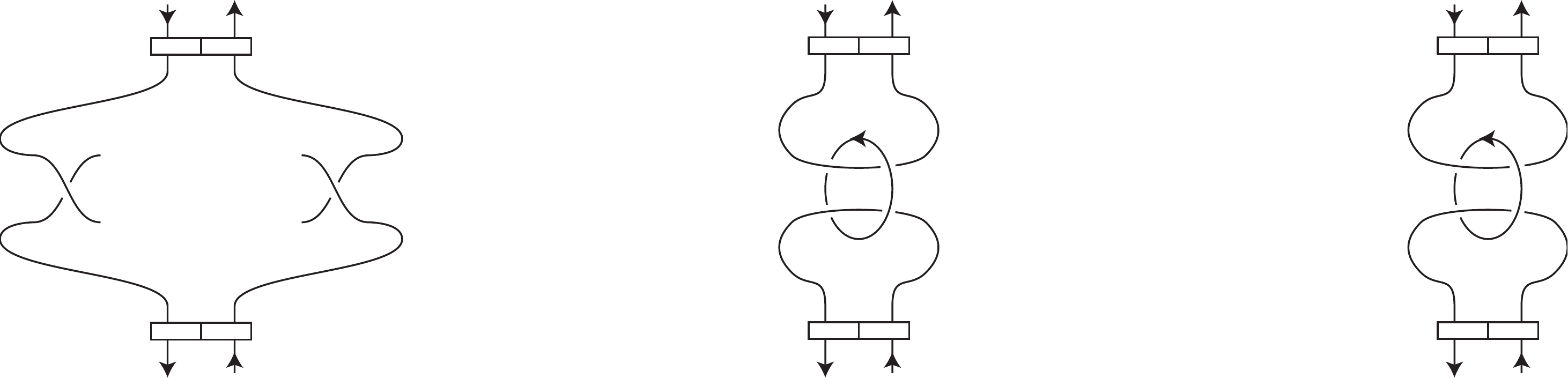}
\caption{}
\label{fig:twist1}
\end{figure}
\noindent
Furthermore, we apply Lemma \ref{lemma:ymn} to two $n$ arcs which go through $R_n$.
By Corollary \ref{cor:vanish}, they vanish except the $(n,n)$-th  $q$-symmetrizer.
Therefore, $T_{n,p}$ is given by
\begin{align*}
T_{n,p} &= (a^nq^{n(n-1)})^{-2p}t_{n,p} \theta_{n,n} x_{n,n}^n \\
&=
 (a^nq^{n(n-1)})^{-2p} \{ n \}! \{2n-2;a\}_{n} (-1)^n \frac{  \{n \}!  }{  \{2n-2;a\}_n }
t_{n,p} \\
&=
 (-1)^n  (a^nq^{n(n-1)})^{-2p} (\{n\}!)^2 t_{n,p} 
\end{align*}

Finally, by comparing both $T_{n,p}$'s, we have the following proposition.
\begin{proposition}
$t_{n,p}$ is given by
\begin{align*}
t_{n,p} 
& = (-1)^n \sum_{i=0}^n   (-1)^i   \frac{( a^i q^{i(i-1)})^{2p} }{ \{ i \}! \{ n-i \}!}.
\frac{\{ 2i-1 ; a \}}{ \{ n+i-1; a \}_{n+1} }
\end{align*}
\end{proposition}

For $|p| \geq 2$, we need the $(n,n)$-th $q$-symmetrizer to calculate the twisting formula.
But for $p= \pm 1$,  we can use Lemma \ref{lemma:alpha}. 
Hence,  by combining Proposition \ref{proposition:omega}, we have the following corollary.
\begin{corollary}
For $p=\pm 1$, $t_{n,p}$ is given by
\begin{align*}
t_{n,1} &= \frac{  a^n q^{ \frac{n(n-1)}{  2  } }  }{  \{n \}!   } = t_n\\
t_{n,-1} &=
(-1)^{n}\frac{  a^{-n} q^{-\frac{n(n-1)}{  2  }}  }{  \{n \}!   } = \bar{t}_n.
\end{align*}
\end{corollary}

\section{Calculations of the colored HOMFLY-PT polynomial}
Now, we are ready to calculate the colored HOMFLY-PT polynomial of twist knots.
The $n$-th colored HOMFLY-PT polynomial for a knot $K$ is given by
\begin{align*}
\mathscr{H}_n(K) = \frac{\langle K(H_n) \rangle   }{ \langle \text{Unknot}(H_n) \rangle } 
=
\frac{ \{n\}!  }  {\{ n-1;a\}_n} {\langle K(H_n) \rangle},
\end{align*}
where $K(H_n)$  stands for $K$ cabled by  $H_n$  with compatible orientations.
$\mathscr{H}_n(K) $ is normalized so that it is to be $1$ for  the unknot. 
We assume that the framing of the knot $K$ is to be $0$.  

For an integer $p$, the twist knot $K_p$ is described in Figure \ref{fig:twistknotKp}.
\begin{figure}[H]
\begin{picture}(0,0)(0,0)
\put(23,17){\scriptsize $\vdots$}
\put(33,23){$\Biggr\}$}
\put(40,24){\scriptsize $p$ full twists}
\end{picture}
\centering
\includegraphics[width=15mm,pagebox=cropbox,clip]{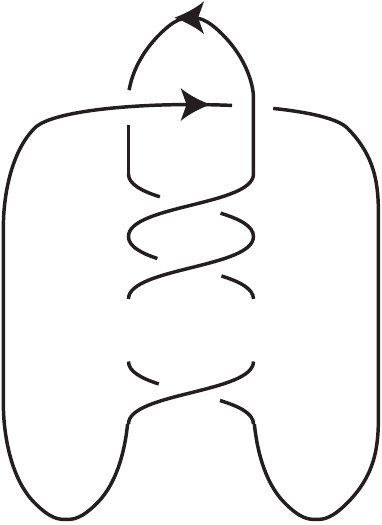}
\caption{The twist knot $K_p$}
\label{fig:twistknotKp}
\end{figure}
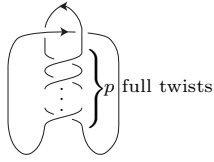
\noindent
Especially, $K_1$ is a left-handed trefoil, and $K_{-1}$ is a figure-eight knot.

Let us calculate $\langle K_p(H_n) \rangle$. 
Some transformations are also due to  \cite{Masbaum}.
Since $H_n$ has the following presentation: 
\begin{align*}
H_n  =  \sum_{i=0}^{n} \frac{ \{ n+i-1;a \}_{n-i} }{ \{ n-i\}! } R_i,
\end{align*}
we insert this presentation along $K_p$ instead of $H_n$,
and we unknot the $p$ full twists with the encircling map by $\omega_{n}^{p}$.
Here, we remark a pair $R_i \in K_{p}(H_n)$ and $R_j \in \omega_{n}^{p}$.
The pair is a cabling of the Whitehead link, of which
one component pierces the other twice in the opposite direction each other.
Moreover, $R_i$ and $R_j$ are a linear combination of $D_{k,k}$ 
with $k \leq i$ and $k \leq j$, respectively.
Therefore, by Corollary \ref{cor:vanish}, 
the pair vanishes if $i \neq j$. 
Hence, we have  $\langle K_p(H_n) \rangle$ as follows:
 \begin{figure}[H]
\begin{picture}(0,0)(0,0)
\put(0,25){$\displaystyle \langle K_p(H_n) \rangle = \sum_{i=0}^{n} \frac{ \{ n+i-1;a \}_{n-i} }{ \{ n-i\}! } t_{i,p}$}
\put(182,6){$R_i$}
\put(208,43){{$R_i$}}
\end{picture}
\centering
\hspace*{56mm}
\includegraphics[width=20mm,pagebox=cropbox,clip]{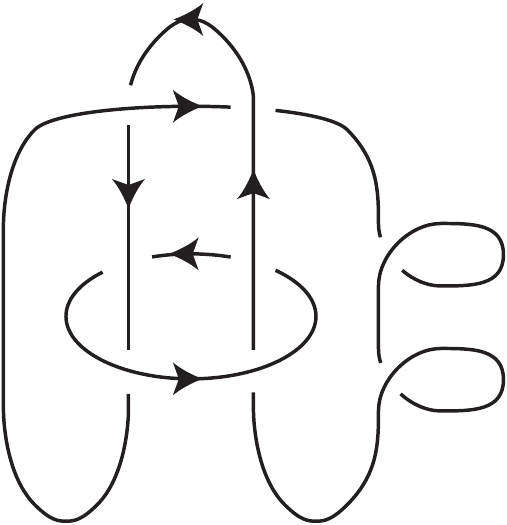}
\label{fig:twistknot0}
\end{figure}

Next, we calculate the part of Figure above. 
The first equality comes from the fact that
$R_i -H_i$ is a linear combination of $R_j$ with $j < i$.
The second equality comes from Lemma \ref{lemma:ymn} and Corollary \ref{cor:vanish}.
\begin{figure}[H]
\begin{picture}(0,0)(0,0)
\put(20,76){$R_i$}
\put(46,113){{$R_i$}}
\put(69,91){$=$}
\put(105,76){$R_i$}
\put(131,113){{$H_i$}}
\put(160,91){$=(a^{i}q^{i(i-1)})^2$}
\put(247,76){$R_i$}
\put(264,119){\scriptsize $i$}
\put(62,6){$R_i$}
\put(-40,21){$=(a^{i}q^{i(i-1)})^2 \alpha_{i,i}^i$}
\put(52,50){\scriptsize $i$}
\put(80,50){\scriptsize $i$}
\end{picture}
\centering
\includegraphics[width=95mm,pagebox=cropbox,clip]{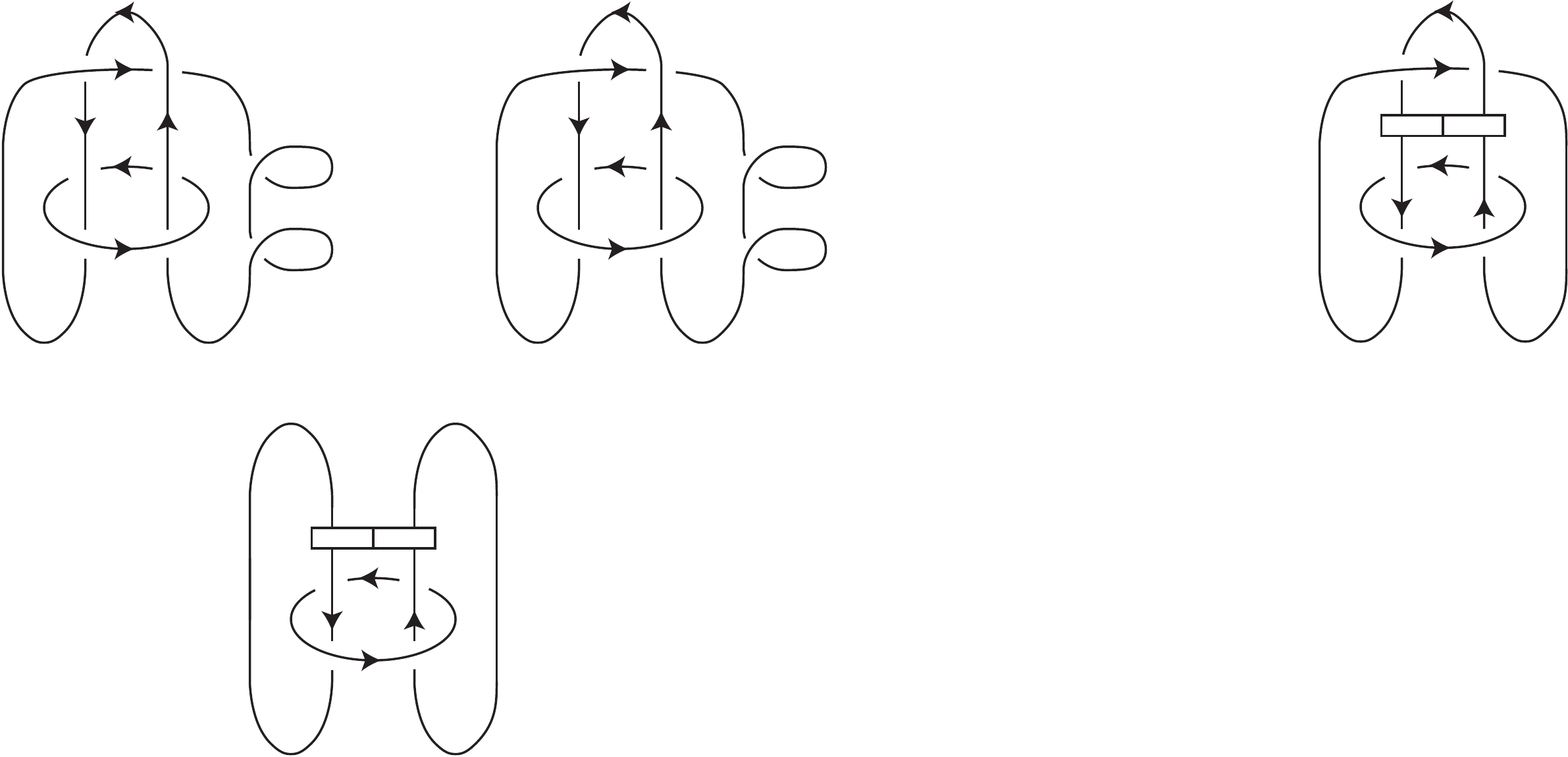}
\begin{align*}
&= (a^{i}q^{i(i-1)})^2 \; \alpha_{i,i}^i \; \theta_{i,i} \; \langle D_{i,i} \rangle \\
&= (a^{i}q^{i(i-1)})^2 (-a)^{-i} q^{-\frac{3}{2}i(i-1)} \{ i \}! \{ i \}! \{ 2i-2 ; a \}_i 
\frac{ \{ 2i -1;a \} ( \{ i-2;a \}_{i-1} )^2  \{ -1;a \}  }{ (\{ i \}!  )^2  } \\
&= (-a)^{i} q^{\frac{  i(i-1)  }{  2  }   }  \{ 2i-1;a \}_{2i}  \{ i-2;a \}_{i}   
\end{align*}
\label{fig:twistknot1}
\end{figure}

Finally, the $i$-th term in the sum of  $\langle K_p(H_n) \rangle$ 
normalized by $\langle \text{Unknot($H_n$)} \rangle$ is expressed by
\begin{align*}
&\frac{  \{ n \}!  }{ \{n-1;a\}_n   }   \times \frac{ \{n+i-1;a\}_{n-i} }{  \{n-i\}!  }  t_{i,p}
\times  (-a)^{i} q^{\frac{  i(i-1)  }{  2  }   }   \{ 2i-1;a \}_{2i}  \{ i-2;a \}_{i} \\
=&
 (-a)^{i} q^{\frac{  i(i-1)  }{  2  }   } \{i\}! t_{i,p}\left[ \begin{array}{@{\,}c@{\,}}n \\ i \end{array} \right]
\{ n+i-1;a \}_i \{i-2;a \}_i.
\end{align*}
Summarizing the above, we obtain the following theorem.
\begin{theorem}
For a twist knot $K_p$, the $n$-th colored HOMFLY-PT polynomial of $K_p$ is given by
\begin{align*}
\mathscr{H}_n(K_p) &=
\sum_{i=0}^n   (-a)^{i} q^{\frac{  i(i-1)  }{  2  }   } \{i\}! t_{i,p}
\left[ \begin{array}{@{\,}c@{\,}}n \\ i \end{array} \right] \{ n+i-1;a \}_i \{i-2;a \}_i \\
&=\sum_{i=0}^n   a^{i} q^{\frac{  i(i-1)  }{  2  }   } \{i\}! s_{i,p}
\left[ \begin{array}{@{\,}c@{\,}}n \\ i \end{array} \right] \{ n+i-1;a \}_i \{i-2;a \}_i, 
\end{align*}
where 
\begin{align*}
s_{i,p} 
& = \sum_{k=0}^i   (-1)^k   \frac{( a^k q^{k(k-1)})^{2p} }{ \{ k \}! \{ i-k \}!}.
\frac{\{ 2k-1 ; a \}}{ \{ i+k-1; a \}_{i+1} }.
\end{align*}
\end{theorem}

\begin{corollary} For the left-handed trefoil $3_1$ and the figure-eight knot $4_1$,
 the $n$-th colored HOMFLY-PT polynomial is given by
\begin{align*}
\mathscr{H}_n(3_1) &= \mathscr{H}_n(K_1) = \sum_{i=0}^n 
(-1)^i a^{2i} q^{i(i-1)}
\left[ \begin{array}{@{\,}c@{\,}}n \\ i \end{array} \right]
\{ n+i-1;a \}_i \{i-2;a \}_i, \\
\mathscr{H}_n(4_1) &= \mathscr{H}_n(K_{-1}) = \sum_{i=0}^n 
\left[ \begin{array}{@{\,}c@{\,}}n \\ i \end{array} \right]
\{ n+i-1;a \}_i \{i-2;a \}_i.
\end{align*}
\end{corollary}

%

\begin{acknowledgement}
The author would like to thank Thang T. Q. Le  and Satoshi Nawata for useful comments.
This work stared from their comments.
\end{acknowledgement}


\end{document}